\documentclass[11pt]{article}

\usepackage{amsmath, amsfonts, amssymb, amsthm,  graphicx, mathtools, enumerate}

\usepackage[blocks, affil-it]{authblk}

\usepackage[numbers, square]{natbib}
\usepackage[CJKbookmarks=true,
            bookmarksnumbered=true,
			bookmarksopen=true,
			colorlinks=true,
			citecolor=red,
			linkcolor=blue,
			anchorcolor=red,
			urlcolor=blue]{hyperref}
\usepackage[usenames]{color}

\usepackage[letterpaper, left=1.2truein, right=1.2truein, top = 1.2truein, bottom = 1.2truein]{geometry}
\usepackage[ruled, vlined, lined, commentsnumbered]{algorithm2e}

\usepackage{prettyref,soul}

\newtheorem{lemma}{Lemma}[section]
\newtheorem{proposition}{Proposition}[section]
\newtheorem{thm}{Theorem}[section]
\newtheorem{definition}{Definition}[section]

\newrefformat{eq}{(\ref{#1})}
\newrefformat{chap}{Chapter~\ref{#1}}
\newrefformat{sec}{Section~\ref{#1}}
\newrefformat{algo}{Algorithm~\ref{#1}}
\newrefformat{fig}{Fig.~\ref{#1}}
\newrefformat{tab}{Table~\ref{#1}}
\newrefformat{rmk}{Remark~\ref{#1}}
\newrefformat{clm}{Claim~\ref{#1}}
\newrefformat{def}{Definition~\ref{#1}}
\newrefformat{cor}{Corollary~\ref{#1}}
\newrefformat{lmm}{Lemma~\ref{#1}}
\newrefformat{lemma}{Lemma~\ref{#1}}
\newrefformat{prop}{Proposition~\ref{#1}}
\newrefformat{app}{Appendix~\ref{#1}}
\newrefformat{ex}{Example~\ref{#1}}
\newrefformat{exer}{Exercise~\ref{#1}}
\newrefformat{soln}{Solution~\ref{#1}}
\newrefformat{cond}{Condition~\ref{#1}}

\def\text#1{\mbox{\rm #1}}

\DeclarePairedDelimiter{\ceil}{\lceil}{\rceil}

\newcommand{\argmax}{\mathop{\rm argmax}}

\newcommand{\norm}[1]{\|{#1} \|}

\newcommand{\wh}{\widehat}

\newcommand{\fnorm}[1]{\|#1\|_{\rm F}}
\newcommand{\nunorm}[1]{\|#1\|_{\rm N}}
\newcommand{\opnorm}[1]{\|#1\|_{\rm op}}

\newcommand{\rank}{\mathop{\sf rank}}
\newcommand{\Tr}{\mathop{\sf Tr}}

\newcommand{\supp}{{\rm supp}}
\newcommand{\iprod}[2]{\left \langle #1, #2 \right\rangle}

\newcommand{\D}{\mathcal{D}}
\newcommand{\U}{\mathcal{U}}
\newcommand{\TV}{{\sf TV}}

\title{Robust Regression via Mutivariate Regression Depth
}
\author{Chao Gao}
\affil{
University of Chicago
}

\begin{document}
\maketitle

\begin{abstract}
This paper studies robust regression in the settings of Huber's $\epsilon$-contamination models. We consider estimators that are maximizers of multivariate regression depth functions. These estimators are shown to achieve minimax rates in the settings of $\epsilon$-contamination models for various regression problems including nonparametric regression, sparse linear regression, reduced rank regression, etc. We also discuss a general notion of depth function for linear operators that has potential applications in robust functional linear regression.
\smallskip

\end{abstract}


\section{Introduction}

Regression is probably one of the most important subjects in statistics. The goal is to learn the conditional mean or median of a response $Y\in\mathbb{R}^m$ given a covariate $X\in\mathbb{R}^p$. Its form ranges from classical low-dimensional linear regression to modern nonparametric and high-dimensional models. In this paper, we study robust regression in the setting of Huber's $\epsilon$-contamination model \cite{huber1964robust}. Namely, consider i.i.d. observations
\begin{equation}
(X_1,Y_1),...,(X_n,Y_n)\sim (1-\epsilon)P_B+\epsilon Q.\label{eq:ecm}
\end{equation}
The distribution $P_B$ models the relation between $X$ and $Y$ via the regression parameter $B$, and $Q$ is an unknown contamination distribution. We need to learn the regression parameter $B$. In this setting, there are approximately $\epsilon n$ observations sampled from $Q$ that do not carry any information about $B$. Since we do not know which observation is contaminated or not, a procedure to recover $B$ must be robust. To be specific, this paper covers the following list of robust regression problems:
\begin{enumerate}
\item \textit{Nonparametric Regression.} The relation between $x$ and $y$ is characterized by $y|x\sim N(f(x),1)$ with some nonparametric function $f$. The goal is to estimate $f$ using data sampled from $(1-\epsilon)P_f+\epsilon Q$.
\item \textit{Sparse Linear Regression.} For a scalar response $y$ and a vector covariate $X$, a linear model is specified by $y|X\sim N(\beta^TX,\sigma^2)$, with some regression vector $\beta$ assumed to be sparse. The goal is to estimate $\beta$ with samples from $(1-\epsilon)P_{\beta}+\epsilon Q$.
\item \textit{Gaussian Graphical Model.} In this setting, we observe i.i.d. samples from $(1-\epsilon)N(0,\Omega^{-1})+\epsilon Q$. The goal is to estimate the sparse precision matrix $\Omega$. The sparsity pattern of $\Omega$ characterizes the graphical model of conditional dependence. The Gaussian graphical model is closely related and can be solved by sparse linear regression \citep{meinshausen2006high}.
\item \textit{Low-Rank Trace Regression.} For a scalar response $y$ and a matrix covariate $X$, a linear model is specified by $y|X\sim N(\Tr(B^TX),\sigma^2)$. The regression matrix $B$ is assumed to be low-rank, and the goal is to estimate it with samples from $(1-\epsilon)P_B+\epsilon Q$.
\item \textit{Multivariate Linear Regression}. In this setting, the response is also multivariate. The linear model is specified by $Y|X\sim N(B^TX,\sigma^2I_m)$. The problem is also termed as multi-task learning. We will show that even there is no relation between the $m$ univariate linear models, estimation of the $m$ columns of $B$ must be done in a joint fashion once the samples are from $(1-\epsilon)P_B+\epsilon Q$.
\item \textit{Multivariate Linear Regression with Group Sparsity.} We consider the same model in the last item, and assume that only a subset of the rows of the regression matrix $B$ are nonzero.
\item \textit{Reduced Rank Regression.} In the same setting of multivariate linear regression, we further assume the regression matrix $B$ is low-rank.
\end{enumerate}

Though the seven problems listed are very different, and the regression parameter we want to recover can be a vector, a matrix or even a function, we consider a unified robust estimation procedure in this paper. In the setting of multivariate linear regression, we use $\mathbb{P}$ to denote the joint distribution of $X\in\mathbb{R}^p$ and $Y\in\mathbb{R}^m$. The multivariate regression depth of $B\in\mathbb{R}^{p\times m}$ is defined as
\begin{equation}
\D_{\U}(B,\mathbb{P})=\inf_{U\in\U}\mathbb{P}\left\{\iprod{U^TX}{Y-B^TX}\geq 0\right\},\label{eq:depth}
\end{equation}
for some subset $\U\subset\mathbb{R}^{p\times m}\backslash\{0\}$. The definition of multivariate regression depth in the form of (\ref{eq:depth}) first appeared in \cite{mizera2002depth}. A very similar but earlier definition was proposed in \cite{bern2000multivariate}. When $m=1$, this is reduced to the univariate regression depth in \cite{rousseeuw1999regression}. When observations are sampled from (\ref{eq:ecm}), a robust estimator for $B$ is defined as the maximizer of the empirical depth function. That is, $\wh{B}=\argmax_{B\in\mathcal{B}}\D_{\U}(B,\mathbb{P}_n)$, where $\mathbb{P}_n$ is the empirical measure of (\ref{eq:ecm}). With various choices of $\mathcal{B}$ and $\mathcal{U}$, we are able to estimate the regression parameters of all the seven problems listed above. The error rates are proved to be minimax optimal under the $\epsilon$-contamination model.

The $\epsilon$-contamination model was first proposed by Peter Huber \citep{huber1964robust}. Its properties have been studied by \cite{huber1965robust,huber1973minimax,bickel1984robust,donoho2015variance} among others. Most early works studied $Q$ with some assumptions. Some recent papers considered the $\epsilon$-contamination model with $Q$ allowed to be any distribution. To be specific, \cite{chen2015robust,chen2015general} showed that the minimax rate of recovering a parameter under the $\epsilon$-contamination model takes a unified formula $\mathcal{R}(\epsilon)\asymp \mathcal{R}(0)\vee \omega(\epsilon,\Theta,L)$. In other words, the minimax rate is determined by two terms. The first term $\mathcal{R}(0)$ is the minimax rate without contamination, and $\omega(\epsilon,\Theta,L)$ is an extra term caused by contamination, where $\epsilon$ is the contamination proportion, $\Theta$ is the parameter space, and $L$ is the loss function of the problem. Despite the progress of fundamental limits, efficient algorithms of estimation in $\epsilon$-contamination models are usually very hard to find. A very recent paper \cite{lai2016agnostic} proposed an algorithm for estimating multivariate mean. The error rate is nearly minimax only when the covariance matrix is known. Given the hardness of computational issues, we will study computationally efficient robust regression algorithms under $\epsilon$-contamination models in a separate paper.

Robust regression is a popular subject in statistics. However, most papers studied robust regression without considering an $\epsilon$ fraction of contamination \citep{huber1973robust,siegel1982robust,rousseeuw1984robust,leroy1987robust,fan2014robust}. The paper \cite{loh2015high} considered contamination, but in a different form from (\ref{eq:ecm}). Thus, the performance of many proposed procedures in the literature have not been tested under (\ref{eq:ecm}). An example in \cite{chen2015robust} shows that even procedures with high breakdown points may not achieve the optimal rate of the $\epsilon$-contamination model. Conversely, \cite{chen2015robust} also shows that a good performance under the $\epsilon$-contamination model must imply a high breakdown point. This serves as the main motivation to study robust regression using $\epsilon$-contamination models. Though sparse linear regression and low-rank trace regression have already been studied in \cite{chen2015general} under the $\epsilon$-contamination model, the proposed procedure of \cite{chen2015general} is based on robust testing and thus requires the assumption that the regression vector or matrix must have bounded $\ell_2$ or Frobenius norm. In contrast, the estimator obtained by maximizing the regression depth does not require this assumption to achieve rate-optimality.

The rest of the paper is organized as follows. Section \ref{sec:review} reviews the definition and properties of the multivariate regression depth function. The applications in robust regression with one response variable are studied in Section \ref{sec:uni}. The applications in multivariate robust regression are studied in Section \ref{sec:multi}. Section \ref{sec:disc} discusses some extensions of the results for elliptical distributions. A general notion of regression depth for learning linear operators is also discussed in that section. All technical proofs are given in Section \ref{sec:pf}.

We close this section by introducing the notation used in the paper. For $a,b\in\mathbb{R}$, let $a\vee b=\max(a,b)$ and $a\wedge b=\min(a,b)$. For an integer $m$, $[m]$ denotes the set $\{1,2,...,m\}$. Given a set $S$, $|S|$ denotes its cardinality, and $\mathbb{I}_S$ is the associated indicator function. For two positive sequences $\{a_n\}$ and $\{b_n\}$, the relation $a_n\lesssim b_n$ means that $a_n\leq Cb_n$ for some constant $C>0$, and $a_n\asymp b_n$ if both $a_n\lesssim b_n$ and $b_n\lesssim a_n$ hold. For a vector $v\in\mathbb{R}^p$, $\norm{v}$ denotes the $\ell_2$ norm, $\|v\|_1$ the $\ell_1$ norm and $\supp(v)=\{j\in[p]:v_j\neq 0\}$ is its support. For a matrix $A\in\mathbb{R}^{d_1\times d_2}$, $\rank(A)$ denotes its rank, $\text{vec}(A)$ is its vectorization, $\fnorm{A}=\norm{\text{vec}(A)}$ is the matrix Frobenius norm, $\|A\|_{\ell_1}=\max_{1\leq j\leq d_2}\sum_{i=1}^{d_1}|A_{ij}|$ is the matrix $\ell_1$ norm, and the nuclear norm $\|A\|_{\rm N}$ is its largest singular value. When $A$ is an squared matrix, $\Tr(A)$ denotes its trace. For two matrices $A,B\in\mathbb{R}^{d_1\times d_2}$, their trace inner product is $\iprod{A}{B}=\Tr(AB^T)$. For two probability distributions $P_1$ and $P_2$, their total variation distance is $\TV(P_1,P_2)=\sup_B|P_1(B)-P_2(B)|$. We use $\mathbb{P}$ and $\mathbb{E}$ to denote generic probability and expectation whose distribution is determined from the context.

\section{The Multivariate Regression Depth}\label{sec:review}

For a joint probability distribution $\mathbb{P}$ of $X\in\mathbb{R}^p$ and $Y\in\mathbb{R}^m$, the multivariate regression depth of $B\in\mathbb{R}^{p\times m}$ is define in (\ref{eq:depth}). Even for $m$ independent univariate regression problems, the multivariate regression depth treats the $m$ regression problems in a joint way. Later we will see this is essential to achieve optimal rates in Huber's $\epsilon$-contamination models.

The multivariate regression depth function is a special case of tangent depth defined by \cite{mizera2002depth}. A very closely related definition was considered in \cite{bern2000multivariate}. Many important properties of the multivariate regression depth are discussed in \cite{mizera2002depth}. For example, it is invariant with respect to linear transformation when $\U=\mathbb{R}^{p\times m}\backslash\{0\}$ in the sense that for any invertible $G\in\mathbb{R}^{p\times p}$ and $H\in\mathbb{R}^{m\times m}$,
$$\D_{\U}\Big(B,\mathcal{L}(X,Y)\Big)=\D_{\U}\Big(G^{-1}BH^T,\mathcal{L}(GX,HY)\Big),$$
where $\mathcal{L}(\cdot)$ denotes the law. We refer the readers to \cite{mizera2002depth,bern2000multivariate,rousseeuw1999regression,struyf1999halfspace,amenta2000regression} for other important properties.

The general multivariate regression depth function covers some important cases. 
When $m=1$, it is Rousseeuw and Hubert's univariate regression depth \citep{rousseeuw1999regression},
\begin{equation}
\D_{\U}(\beta,\mathbb{P})=\inf_{u\in\U}\mathbb{P}\left\{u^TX(y-X^T\beta)\geq 0\right\}.\label{eq:uni-depth}
\end{equation}
When $p=1$ and the covariate is $1$, it is Tukey's half-space depth \citep{tukey1975mathematics} for multivariate location estimation,
\begin{equation}
\D_{\U}(\theta,\mathbb{P})=\inf_{u\in\U}\mathbb{P}\left\{u^T(Y-\theta)\geq 0\right\}.\label{eq:tukey-depth}
\end{equation}
The error rate of maximizing Tukey's depth under the $\epsilon$-contamination model was studied by \cite{chen2015general}. Our main results for multivariate regression not only cover univariate regression depth, but also reproduce the result of \cite{chen2015general} for Tukey's depth.

Section \ref{sec:uni} and Section \ref{sec:multi} study the error rates of the estimator
\begin{equation}
\wh{B}=\argmax_{B\in\mathcal{B}}\D_{\U}(B,\mathbb{P}_n)\label{eq:general-estimator}
\end{equation}
for univariate and multivariate regression, respectively. To benchmark our main results, we need to introduce the general minimax lower bound for $\epsilon$-contamination models obtained by \cite{chen2015robust}.
\begin{thm}[Chen, Gao \& Ren (2015) \citep{chen2015robust}]\label{thm:lower}
Given a statistical experiment $\{P_{\theta}:\theta\in\Theta\}$ and a loss function $L(\cdot,\cdot)$,
define
$$\omega(\epsilon,\Theta,L)=\sup\left\{L(\theta_1,\theta_2): \TV(P_{\theta_1},P_{\theta_2})\leq \epsilon/(1-\epsilon); \theta_1,\theta_2\in\Theta\right\}.$$
Suppose there is some $\mathcal{R}(0)$ such that
\begin{equation}
\inf_{\hat{\theta}}\sup_{\theta\in\Theta,Q}\mathbb{P}_{(\epsilon,\theta,Q)}\left\{L(\hat{\theta},\theta)\geq\mathcal{R}(\epsilon)\right\}\geq c\label{eq:lower}
\end{equation}
holds for $\epsilon=0$. Then, (\ref{eq:lower}) also holds for any $\epsilon\in(0,1)$ with $\mathcal{R}(\epsilon)\asymp\mathcal{R}(0)\vee\omega(\epsilon,\Theta)$. The notation $\mathbb{P}_{(\epsilon,\theta,Q)}$ stands for $(1-\epsilon)P_{\theta}+\epsilon Q$.
\end{thm}
Theorem \ref{thm:lower} gives a general minimax lower bound for parameter estimation in the settings of $\epsilon$-contamination models. The quantity $\omega(\epsilon,\Theta,L)$ is called modulus of continuity \citep{donoho1991geometrizing}, which characterizes the ability of a loss function $L(\cdot,\cdot)$ to distinguish between two parameters whose corresponding probability distributions are $\epsilon/(1-\epsilon)$ close in total variation distance. The rate $\mathcal{R}(\epsilon)\asymp\mathcal{R}(0)\vee\omega(\epsilon,\Theta,L)$ is the best possible one that can be achieved by any procedure. For many loss functions, $\omega(\epsilon,\Theta,L)$ is at the order of $\epsilon^2$. We will show that the estimator induced by the multivariate depth function is able to achieve the rate $\mathcal{R}(\epsilon)\asymp\mathcal{R}(0)\vee\omega(\epsilon,\Theta,L)$ for all the seven regression problems considered in the paper.

\section{Applications of Regression Depth} \label{sec:uni}

\subsection{Nonparametric Regression}

Consider the nonparametric regression model $y=f(x)+z$. To be specific, we use the distribution $P_f$ to denote the sampling process that first sample $x\sim\text{Unif}[0,1]$ and then sample $y|x\sim N(f(x),1)$.
 The regression function admits the expansion $f(x)=\sum_{j=1}^{\infty}\beta_j\phi_j(x)$, where $\{\phi_j\}_{j=1}^{\infty}$ is the Fourier basis on $L^2[0,1]$. We assume the true regression function belongs to the following Sobolev ball:
$$S_{\alpha}(M)=\left\{f=\sum_{j=1}^{\infty}\beta_j\phi_j:\sum_{j=1}^{\infty}j^{2\alpha}\beta_j^2\leq M^2\right\}.$$
The smoothness parameter $\alpha>0$ and radius $M>0$ are assumed as constants throughout the section.

Define the vector of infinite size $X=\{\phi_j(x)\}_{j\in[\infty]}\in\mathbb{R}^{\infty}$. Then, the model becomes $y=\beta^TX+z$. Recovery of $f$ is equivalent to recovery of $\beta\in\mathbb{R}^{\infty}$. Define
$$\U_k=\left\{u\in\mathbb{R}^{\infty}\backslash\{0\}: u_j=0\text{ for all }j>k\right\}.$$
We use the univariate regression depth (\ref{eq:uni-depth}) to estimate the Fourier coefficients $\beta$ by
\begin{equation}
\hat{\beta}=\argmax_{\beta\in\U_k}\D_{\U_{k}}(\beta,\{(X_i,y_i)\}_{i=1}^n).\label{eq:nonpar-est}
\end{equation}
To be specific, the empirical regression depth for this problem is
$$\D_{\U_{k}}(\beta,\{(X_i,y_i)\}_{i=1}^n)=\inf_{u\in\U_k}\frac{1}{n}\sum_{i=1}^n\mathbb{I}\left\{\left(\sum_{j=1}^{\infty}u_j\phi_j(x_i)\right)\left(y_i-\sum_{j=1}^{\infty}\beta_j\phi_j(x_i)\right)\geq 0\right\}.$$
Since the regression function is in the space $S_{\alpha}(M)$, we expect that $\beta_j$'s are negligible for high frequencies, and thus the regression depth does not need to involve frequencies after some level $k$.

We first give a result for the uniform convergence of the empirical regression depth.

\begin{proposition}\label{prop:nonpar}
For any probability measure $\mathbb{P}$ and its associated empirical measure $\mathbb{P}_n$, we have for any $\delta>0$,
$$\sup_{\beta\in\U_k}\left|\D_{\U_{k}}(\beta,\mathbb{P}_n)-\D_{\U_{k}}(\beta,\mathbb{P})\right|\leq C\sqrt{\frac{k}{n}}+\sqrt{\frac{\log(1/\delta)}{2n}},$$
with probability at least $1-2\delta$, where $C>0$ is some absolute constant.
\end{proposition}

Using this result, we can study the convergence rate of the estimator (\ref{eq:nonpar-est}) in the setting of the $\epsilon$-contamination model. Namely, consider i.i.d. observations from $\mathbb{P}_{(\epsilon,f,Q)}=(1-\epsilon)P_f+\epsilon Q$.

\begin{thm}\label{thm:nonpar}
Consider the estimator $\hat{f}=\sum_j\hat{\beta}_j\phi_j$ with $k=\ceil{n^{\frac{1}{2\alpha+1}}}$. Assume that $\epsilon^2+n^{-\frac{2\alpha}{2\alpha+1}}$ is sufficiently small. Then, we have
$$\|\hat{f}-f\|^2=\int_0^1(\hat{f}(x)-f(x))^2dx\leq C\left(n^{-\frac{2\alpha}{2\alpha+1}}\vee\epsilon^2\right),$$
with $\mathbb{P}_{(\epsilon,f,Q)}$-probability at least $1-\exp\left(-C'(n^{\frac{1}{2\alpha+1}}+n\epsilon^2)\right)$ uniformly over all $Q$ and $f\in S_{\alpha}(M)$, where $C,C'$ are some absolute constants.
\end{thm}

The rate consists of two terms. The first term $n^{-\frac{2\alpha}{2\alpha+1}}$ is the classical minimax rate for nonparametric function estimation in the space $S_{\alpha}(M)$. See \cite{tsybakov2008introduction,johnstone2011gaussian} for details. The second term $\epsilon^2$ characterizes the influence of contamination. It is not hard to check that the modulus of continuity for the loss $\|\cdot\|^2$ is of order $\epsilon^2$. Thus, the rate $n^{-\frac{2\alpha}{2\alpha+1}}\vee\epsilon^2$ is minimax optimal by Theorem \ref{thm:lower}.

Given that the minimax rate is $n^{-\frac{2\alpha}{2\alpha+1}}\vee\epsilon^2$, a necessary and sufficient condition to achieve the rate $n^{-\frac{2\alpha}{2\alpha+1}}$ as if there is no contamination is $\epsilon\lesssim n^{-\frac{\alpha}{2\alpha+1}}$. Hence, in order to achieve the minimax rate for $\epsilon=0$, a rate-optimal robust estimator can tolerate at most $n\epsilon\lesssim n^{\frac{\alpha+1}{2\alpha+1}}$ contaminated observations. The number $n^{-\frac{\alpha}{2\alpha+1}}$ can be interpreted as the order of the minimax-rate breakdown point, because the minimax rate will change from $n^{-\frac{2\alpha}{2\alpha+1}}$ to $\epsilon^2$ as soon as $\epsilon\gtrsim n^{-\frac{\alpha}{2\alpha+1}}$. It is interesting to note that a larger $\alpha$ implies a smaller order of $n^{-\frac{\alpha}{2\alpha+1}}$.

\subsection{Sparse Linear Regression}\label{sec:sparse}

Consider the sparse linear regression model, where the response and covariate are linked by the equation $y=\beta^TX+\sigma z$.
The regression vector $\beta$ is assumed to belong to the sparse set:
\begin{equation}
\Theta_s=\left\{\beta\in\mathbb{R}^p\backslash\{0\}: \sum_{j=1}^p\mathbb{I}\{\beta_j\neq 0\}\leq s\right\}.\label{eq:sparse-set}
\end{equation}
The joint distribution $(X,y)\sim P_{\beta}$ is specified by the sampling process $X\sim N(0,\Sigma)$ and $y|X\sim N(\beta^TX,\sigma^2)$. For simplicity of notation, we suppress the dependence on $\Sigma$ and $\sigma^2$ for $P_{\beta}$.

Using the univariate regression depth function (\ref{eq:uni-depth}), we define a sparse estimator by
\begin{equation}
\hat{\beta}=\argmax_{\beta\in\Theta_s}\D_{\Theta_{2s}}(\beta,\{(X_i,y_i)\}_{i=1}^n).\label{eq:sparse-est}
\end{equation}
We take advantage of the sparsity of the problem by setting $\U=\Theta_{2s}$ and $\mathcal{B}=\Theta_s$ in (\ref{eq:general-estimator}). For this sparse regression depth, its uniform convergence property is given by the following proposition.

\begin{proposition}\label{prop:sparse}
For any probability measure $\mathbb{P}$ and its associated empirical measure $\mathbb{P}_n$, we have for any $\delta>0$,
$$\sup_{\beta\in\Theta_s}\left|\D_{\Theta_{2s}}(\beta,\mathbb{P}_n)-\D_{\Theta_{2s}}(\beta,\mathbb{P})\right|\leq C\sqrt{\frac{s\log\left(\frac{ep}{s}\right)}{n}}+\sqrt{\frac{\log(1/\delta)}{2n}},$$
with probability at least $1-2\delta$, where $C>0$ is some absolute constant.
\end{proposition}

Before giving the convergence rate of (\ref{eq:sparse-est}), we need to define the following quantity:
$$\kappa=\inf_{|\supp(v)|=2s}\frac{\|\Sigma^{1/2}v\|}{\|v\|}.$$
This is called restricted eigenvalue in sparse linear regression literature. Now we are ready for the main results. Consider i.i.d. observations from $\mathbb{P}_{(\epsilon,\beta,Q)}=(1-\epsilon)P_{\beta}+\epsilon Q$.

\begin{thm}\label{thm:sparse}
Consider the estimator $\hat{\beta}$.
Assume that $\epsilon^2+\frac{s\log\left(\frac{ep}{s}\right)}{n}$ is sufficiently small. Then, we have
\begin{eqnarray}
\label{eq:sparse-pred}\|\hat{\beta}-\beta\|_{\Sigma}^2=\|\Sigma^{1/2}(\hat{\beta}-\beta)\|^2 &\leq& C\sigma^2\left(\frac{s\log\left(\frac{ep}{s}\right)}{n}\vee\epsilon^2\right), \\
\label{eq:sparse-l2}\|\hat{\beta}-\beta\|^2 &\leq& C\frac{\sigma^2}{\kappa^2}\left(\frac{s\log\left(\frac{ep}{s}\right)}{n}\vee\epsilon^2\right), \\
\label{eq:sparse-l1}\|\hat{\beta}-\beta\|_1^2 &\leq& C\frac{\sigma^2}{\kappa^2}\left(\frac{s^2\log\left(\frac{ep}{s}\right)}{n}\vee s\epsilon^2\right),
\end{eqnarray}
with $\mathbb{P}_{(\epsilon,\beta,Q)}$-probability at least $1-\exp\left(-C'\left(s\log\left(\frac{ep}{s}\right)+n\epsilon^2\right)\right)$ uniformly over all $Q$ and $\beta\in\Theta_s$, where $C,C'$ are some absolute constants.
\end{thm}

The rates are given in prediction loss, squared $\ell_2$ loss and squared $\ell_1$ loss, respectively. The rate for the prediction loss does not depend on the covariance of the covariates $\Sigma$. On the other hand, the rates for the squared $\ell_2$ loss and the squared $\ell_1$ loss depend on $\Sigma$ through a $\kappa^{-2}$ factor.

These rates were also obtained by \cite{chen2015general} under the $\epsilon$-contamination model with a testing-based estimator. However, their results only hold for a subset of $\Theta_s$. In particular, they need to further impose two extra assumptions that $\|\beta\|$ is bounded by the order of $\sigma/\kappa$ and the largest $2s$-sparse eigenvalue of $\Sigma$ is at the order of $\kappa$. In contrast, Theorem \ref{thm:sparse} removes these two assumptions and the convergence rates hold uniformly for all $\beta\in\Theta_s$.

When $\epsilon=0$, the rates obtained in Theorem \ref{thm:sparse} are all minimax optimal by \cite{ye2010rate,raskutti2011minimax,verzelen2012minimax}. Though most lower bound results in the literature for sparse linear regression are for fixed design. They can be easily modified into the random design setting considered here. The details are referred to the related discussion in \cite{raskutti2011minimax,chen2015general}.

For a general $\epsilon>0$, it is direct to check that
\begin{eqnarray*}
\omega\left(\epsilon,\Theta_s,\|\cdot\|^2_{\Sigma}\right) &\asymp& \sigma^2\epsilon^2, \\
\omega\left(\epsilon,\Theta_s,\|\cdot\|^2\right) &\asymp& \frac{\sigma^2\epsilon^2}{\kappa^2}, \\
\omega\left(\epsilon,\Theta_s,\|\cdot\|^2_1\right) &\asymp& \frac{s\sigma^2\epsilon^2}{\kappa^2}.
\end{eqnarray*}
Thus, by Theorem \ref{thm:lower}, the rates are also minimax optimal for $\epsilon>0$.

Theorem \ref{thm:sparse} and the minimax lower bound of the problem shows that the minimax rates are determined by the trade-off between $\frac{s\log\left(\frac{ep}{s}\right)}{n}$ and $\epsilon^2$. When $\epsilon^2\lesssim\frac{s\log\left(\frac{ep}{s}\right)}{n}$, the term $\frac{s\log\left(\frac{ep}{s}\right)}{n}$ dominates, and the minimax rates are the same as those for $\epsilon=0$. In this regime, the contamination has no effect on the minimax rates. Note that $\epsilon\lesssim\frac{s\log\left(\frac{ep}{s}\right)}{n}$ means that a rate-optimal estimator is able to tolerate at most $n\epsilon\lesssim \sqrt{ns\log\left(\frac{ep}{s}\right)}$ contaminated observations before the minimax rate is changed. It is interesting that $\sqrt{ns\log\left(\frac{ep}{s}\right)}$ is an increasing function of the sparsity level $s$. Similar remarks also apply to the other regression problems considered in the paper.

\subsection{Gaussian Graphical Model}\label{sec:GGM}

In this section, we consider the Gaussian graphical model $P_{\Omega}= N(0,\Omega^{-1})$. The precision matrix $\Omega$ belongs to the following sparse class:
$$\mathcal{F}_s(M)=\left\{\Omega=\Omega^T\in\mathbb{R}^{p\times p}:M^{-1}\leq \lambda_{\min}(\Omega)\leq\lambda_{\max}(\Omega)\leq M, \max_{1\leq i\leq p}\sum_{j=1}^p\mathbb{I}\{\Omega_{ij}\neq 0\}\leq s\right\}.$$
The notation $\lambda_{\min}(\cdot)$ and $\lambda_{\max}(\cdot)$ stand for the smallest and the largest eigenvalues.
This class was previously considered in \cite{ren2015asymptotic}. We assume the number $M$ is a constant throughout this section.

For a random vector $X\sim N(0,\Omega^{-1})$, the sparsity pattern of $\Omega$ characterizes the graphical model of conditional dependence. In particular, $\Omega_{ij}=0$ if and only if $X_i$ is independent of $X_j$ given all remaining variables.

Moreover, there is simple linear model that links $X_j$ and $X_{-j}$, where we use $X_{-j}$ to denote the $(p-1)$-dimensional subvector of $X$ excluding the $j$th variable.
Define $\beta_{(j)}=-\Omega_{jj}^{-1}\Omega_{-j,j}$, then
\begin{equation}
X_j=\beta_{(j)}^TX_{-j}+\xi_j,\label{eq:ggm-lm}
\end{equation}
where the noise has distribution $\xi_j\sim N(0,\Omega_{jj}^{-1})$ and is independent of $X_{-j}$. Methods based on (\ref{eq:ggm-lm}) are proposed in the literature to estimator $\Omega$. See \cite{meinshausen2006high,yuan2010high} for some examples.

With i.i.d. observations from $\mathbb{P}_{(\epsilon,\Omega,Q)}$, we discuss how to explore the linear model (\ref{eq:ggm-lm}) to estimate the precision matrix $\Omega$ in a robust way. For each $j\in[n]$, we need to estimate $\beta^{(j)}$ and the variance of $\xi_j$, which is $\Omega_{jj}^{-1}$, respectively. Without loss of generality, assume the sample size $n$ is even. We split the data into two halves. We use the first half to estimate $\beta_{(j)}$ by
$$\hat{\beta}_{(j)}=\argmax_{\beta\in\Theta_s}\D_{\Theta_{2s}}(\beta,\{(X_{-j,i},X_{ji})\}_{i=1}^{n/2}).$$
The set $\Theta_s$ is defined in (\ref{eq:sparse-set}) with the dimension $p$ replaced by $p-1$. The convergence rate of $\hat{\beta}_{(j)}$ is given by Theorem \ref{thm:sparse}.
We then use the second half of the data together with $\hat{\beta}_{(j)}$ to estimate the variance of $\xi_j$. For each $i=n/2+1,...,n$, define the residue
$$w_{ji}=(X_{ji}-\hat{\beta}_{(j)}^TX_{-j,i})^2.$$
Then, we estimate $\Omega_{jj}^{-1}$ by median absolute deviation,
$$\wh{\Omega}_{jj}^{-1}=\frac{\text{Median}(\{w_{ji}\}_{i=n/2+1}^n)}{[\Phi^{-1}(3/4)]^2},$$
where $\Phi(\cdot)$ is the cumulative distribution function of $N(0,1)$.
The convergence rate of $\wh{\Omega}_{jj}^{-1}$ is given by \cite{chen2015robust}.

Now we are ready to define the estimator of the precision matrix $\Omega$ by assembling all pieces. For the $j$th column, its $j$th entry is estimated by $\wh{\Omega}_{jj}$. The remaining entries are estimated by $\wh{\Omega}_{-j,j}=-\wh{\Omega}_{jj}\hat{\beta}_{(j)}$. The convergence rate of the estimator $\wh{\Omega}$ is given by the following theorem.

\begin{thm}\label{thm:GGM}
Consider the estimator $\wh{\Omega}$.
Assume that $\epsilon^2+\frac{s\log\left(\frac{ep}{s}\right)}{n}$ is sufficiently small. Then, we have
$$\|\wh{\Omega}-\Omega\|_{\ell_1}^2\leq C\left(\frac{s^2\log\left(\frac{ep}{s}\right)}{n}\vee s\epsilon^2\right),$$
with $\mathbb{P}_{(\epsilon,\Omega,Q)}$-probability at least $1-\exp\left(-C'\left(s\log\left(\frac{ep}{s}\right)+n\epsilon^2\right)\right)$ uniformly over all $Q$ and $\Omega\in\mathcal{F}_s(M)$, where $C,C'$ are some absolute constants.
\end{thm}

Theorem \ref{thm:GGM} gives the error rate of $\wh{\Omega}$ in terms of squared matrix $\ell_1$ norm. Note that the estimator $\wh{\Omega}$ may not be positive semidefinite. A simple projection step discussed in \cite{yuan2010high} leads to a positive semidefinite estimator with the same error rate.

The minimax lower bound of the problem is given by the following theorem.
\begin{thm}\label{thm:lowerGMM}
Assume $p>c_1n^{\beta}$ for some constants $\beta>1$ and $c_1>0$, and $\frac{s^2(\log p)^3}{n}$ is sufficiently small. Then,
$$\inf_{\wh{\Omega}}\sup_{\Omega\in\mathcal{F}_s(M),Q}\mathbb{P}_{(\epsilon,\Omega,Q)}\left\{\|\wh{\Omega}-\Omega\|_{\ell_1}^2> C\left(\frac{s^2\log p}{n}\vee s\epsilon^2\right)\right\}\geq c,$$
for some constants $C,c>0$.
\end{thm}
\begin{proof}
By Theorem \ref{thm:lower}, the minimax lower is in the form of $\mathcal{R}(0)\vee\omega\left(\epsilon,\mathcal{F}_s(M),\|\cdot\|^2_{\ell_1}\right)$. The first term $\mathcal{R}(0)$ has order $\frac{s^2\log p}{n}$, which was proved in \cite{cai2012estimating}. Direct calculation gives the order of the second term $\omega\left(\epsilon,\mathcal{F}_s(M),\|\cdot\|^2_{\ell_1}\right)\asymp s\epsilon^2$.
\end{proof}
Combining the conclusions of Theorem \ref{thm:GGM} and Theorem \ref{thm:lowerGMM}, we conclude that the minimax rate for estimating $\Omega$ under the squared matrix $\ell_1$ norm in the setting of $\epsilon$-contamination model is $\frac{s^2\log p}{n}\vee s\epsilon^2$. Moreover, the estimator $\wh{\Omega}$ based on regression depth is able to achieve the minimax rate.

\subsection{Low-Rank Trace Regression}\label{sec:rank}

The goal of trace regression is to recover a low-rank matrix $B\in\mathbb{R}^{p_1\times p_2}$ from noisy linear observations specified by the model $y=\Tr(X^TB)+\sigma z$. We denote by $P_B$ the joint distribution of $X\in\mathbb{R}^{p_1\times p_2}$ and $y\in\mathbb{R}$ that follows $\text{vec}(X)\sim N(0,\Sigma)$ and $y|X\sim N(\Tr(X^TB),\sigma^2)$. Again, the dependence on $\Sigma$ and $\sigma^2$ is suppressed for the notation $P_B$.
The matrix $B$ is assumed to belong to the following set:
\begin{equation}
\mathcal{A}_r=\{B\in\mathbb{R}^{p_1\times p_2}\backslash\{0\}:\text{rank}(B)\leq r\}.\label{eq:lrset}
\end{equation}
The univariate regression depth (\ref{eq:uni-depth}) can be easily adapted to the trace regression problem. That is,
$$\D_{\U}(B,\mathbb{P})=\inf_{U\in\U}\mathbb{P}\left\{\iprod{U}{X}(y-\iprod{B}{X})\geq 0\right\},$$
where $\U\subset\mathbb{R}^{p_1\times p_2}$. We take advantage of the low-rank assumption, and define the estimator by
\begin{equation}
\wh{B}=\argmax_{B\in\mathcal{A}_r}\D_{\mathcal{A}_{2r}}(B,\{X_i,y_i\}_{i=1}^n).\label{eq:lr-est}
\end{equation}
We first present a uniform convergence result for regression depth with a low-rank structure.

\begin{proposition}\label{prop:rank}
For any probability measure $\mathbb{P}$ and its associated empirical measure $\mathbb{P}_n$, we have for any $\delta>0$,
$$\sup_{B\in\mathcal{A}_r}\left|\D_{\mathcal{A}_{2r}}(B,\mathbb{P}_n)-\D_{\mathcal{A}_{2r}}(B,\mathbb{P})\right|\leq C\sqrt{\frac{r(p_1+p_2)}{n}}+\sqrt{\frac{\log(1/\delta)}{2n}},$$
with probability at least $1-2\delta$, where $C>0$ is some absolute constant.
\end{proposition}

With the uniform convergence of empirical depth, we can determine the convergence rate of the estimator (\ref{eq:lr-est}). To facilitate the presentation, we define the following quantity:
$$\kappa=\inf_{\text{rank}(A)=2r}\frac{\|\Sigma^{1/2}\text{vec}(A)\|}{\fnorm{A}}.$$
Now, consider i.i.d. observations from $\mathbb{P}_{(\epsilon,B,Q)}=(1-\epsilon)P_B+\epsilon Q$.

\begin{thm}\label{thm:trace}
Consider the estimator $\wh{B}$.
Assume that $\epsilon^2+\frac{r(p_1+p_2)}{n}$ is sufficiently small. Then, we have
\begin{eqnarray*}
\|\Sigma^{1/2}(\text{vec}(\wh{B}-B))\|^2 &\leq& C\sigma^2\left(\frac{r(p_1+p_2)}{n}\vee\epsilon^2\right),\\
\fnorm{\wh{B}-B}^2 &\leq& C\frac{\sigma^2}{\kappa^2}\left(\frac{r(p_1+p_2)}{n}\vee\epsilon^2\right),\\
\nunorm{\wh{B}-B}^2 &\leq& C\frac{\sigma^2}{\kappa^2}\left(\frac{r^2(p_1+p_2)}{n}\vee r\epsilon^2\right),
\end{eqnarray*}
with $\mathbb{P}_{(\epsilon,B,Q)}$-probability at least $1-\exp\left(-C'\left(r(p_1+p_2)+n\epsilon^2\right)\right)$ uniformly over all $Q$ and $B\in\mathcal{A}_r$, where $C,C'$ are some absolute constants.
\end{thm}

Similar to Theorem \ref{thm:sparse}, Theorem \ref{thm:trace} gives rates for prediction loss, squared Frobenius loss and squared nuclear loss, respectively. When $\epsilon=0$, the three rates are all minimax optimal by \cite{koltchinskii2011nuclear}. To see the optimality for $\epsilon>0$, note that
\begin{eqnarray*}
\omega\left(\epsilon,\mathcal{A}_r,\|\cdot\|^2_{\Sigma}\right) &\asymp& \sigma^2\epsilon^2, \\
\omega\left(\epsilon,\mathcal{A}_r,\|\cdot\|^2_{\rm F}\right) &\asymp& \frac{\sigma^2\epsilon^2}{\kappa^2}, \\
\omega\left(\epsilon,\mathcal{A}_r,\|\cdot\|^2_{\rm N}\right) &\asymp& \frac{r\sigma^2\epsilon^2}{\kappa^2}.
\end{eqnarray*}
Thus, by Theorem \ref{thm:lower}, the rates are all minimax optimal.

Results in \cite{chen2015general} gave the same rate for trace regression in the setting of $\epsilon$-contamination model. However, they required extra assumptions such as the boundedness of $\fnorm{B}$ and of $\opnorm{\Sigma}$. In contrast, Theorem \ref{thm:trace} achieves the minimax rate of the problem without these extra assumptions.

\section{Applications of Multivariate Regression Depth}\label{sec:multi}

\subsection{Multivariate Linear Regression}\label{sec:multiple}

Starting from this section, we consider regression problems with multiple responses in the setting of $\epsilon$-contamination model.
Consider the model $Y=B^TX+\sigma Z$, where $B\in\mathbb{R}^{p\times m}$. We use $P_B$ to denote the joint distribution of $X\in\mathbb{R}^p$ and $Y\in\mathbb{R}^m$ specified by $X\sim N(0,\Sigma)$ and $Y|X\sim N(B^TX,\sigma^2 I_m)$. Again, the dependence on $\Sigma$ and $\sigma^2$ is suppressed for the notation $P_B$. We use the multivariate regression depth (\ref{eq:depth}) for estimating $B$.
The estimator is defined as
\begin{equation}
\wh{B}=\argmax_{B\in\mathbb{R}^{p\times m}}\D_{\mathbb{R}^{p\times m}\backslash\{0\}}(B,\{X_i,Y_i\}_{i=1}^n).\label{eq:multiple-est}
\end{equation}
Intuitively, the $m$ univariate regression problems are independent, and one can estimate every column of $B$ separately. The rates are optimal when there is no contamination. However, we will show that this strategy does not lead to rate optimality in the setting of $\epsilon$-contamination model. 

The uniform convergence of the multivariate regression depth is given by the following proposition.

\begin{proposition}\label{prop:multiple}
For any probability measure $\mathbb{P}$ and its associated empirical measure $\mathbb{P}_n$, we have for any $\delta>0$,
$$\sup_{B\in\mathbb{R}^{p\times m}}\left|\D_{\mathbb{R}^{p\times m}\backslash\{0\}}(B,\mathbb{P}_n)-\D_{\mathbb{R}^{p\times m}\backslash\{0\}}(B,\mathbb{P})\right|\leq C\sqrt{\frac{pm}{n}}+\sqrt{\frac{\log(1/\delta)}{2n}},$$
with probability at least $1-2\delta$, where $C>0$ is some absolute constant.
\end{proposition}

Then, define the quantity
\begin{equation}
\kappa=\inf_{v\neq 0}\frac{\|\Sigma^{1/2}v\|}{\|v\|}.\label{eq:multiple-kappa}
\end{equation}
With Proposition \ref{prop:multiple} and the definition of $\kappa$, we are ready to present the main result. Consider i.i.d. observations from $\mathbb{P}_{(\epsilon,B,Q)}=(1-\epsilon)P_B+\epsilon Q$.

\begin{thm}\label{thm:multiple}
Consider the estimator $\wh{B}$.
Assume that $\epsilon^2+\frac{pm}{n}$ is sufficiently small. Then, we have
\begin{eqnarray}
\label{eq:multiple-pred}\Tr((\wh{B}-B)^T\Sigma(\wh{B}-B)) &\leq& C\sigma^2\left(\frac{pm}{n}\vee\epsilon^2\right), \\
\label{eq:multiple-frob}\fnorm{\wh{B}-B}^2 &\leq& C\frac{\sigma^2}{\kappa^2}\left(\frac{pm}{n}\vee\epsilon^2\right),
\end{eqnarray}
with $\mathbb{P}_{(\epsilon,B,Q)}$-probability at least $1-\exp\left(-C'\left(pm+n\epsilon^2\right)\right)$ uniformly over all $Q$ and $B\in\mathbb{R}^{p\times m}$, where $C,C'$ are some absolute constants.
\end{thm}

We first remark that the rates for both prediction loss and squared Frobenius loss are minimax optimal. This can be easily seen from Theorem \ref{thm:lower} and classical multivariate regression results in the literature.

One can also use univariate regression depth to estimate each column of $B$ separately. This leads to the rates $\sigma^2\left(\frac{pm}{n}\vee(m\epsilon^2)\right)$ and $\frac{\sigma^2}{\kappa^2}\left(\frac{pm}{n}\vee(m\epsilon^2)\right)$ for the two loss functions, respectively. Both rates are clearly sub-optimal because of the extra factor of $m$ before $\epsilon^2$. Therefore, in the setting of $\epsilon$-contamination model, even when there is no structural dependence between the columns of $B$, the matrix $B$ needs to be estimated jointly. 

When $p=1$, the multivariate regression depth is closely related to Tukey's halfspace depth (\ref{eq:tukey-depth}). The rate of convergence of Tukey's median was studied by \cite{chen2015robust} in the setting of $\epsilon$-contamination model. Theorem \ref{thm:multiple} can be viewed as an extension of their result for $p>1$.

\subsection{Multivariate Linear Regression with Group Sparsity}

We extend the multivariate regression problem $Y=B^TX+\sigma Z$ to a group sparse setting. The regression matrix $B$ is assumed to belong to the following space
$$\Xi_s=\left\{B\in\mathbb{R}^{p\times m}\backslash\{0\}: \sum_{j=1}^p\mathbb{I}\{B_{j*}\neq 0\}\leq s\right\}.$$
We take advantage of the group sparse structure and define the estimator by
$$\wh{B}=\argmax_{B\in\Xi_s}\D_{\Xi_{2s}}(B,\{(X_i,Y_i)\}_{i=1}^n).$$

The uniform convergence of the multivariate regression depth with group sparse structure is given by the following proposition.
\begin{proposition}\label{prop:group}
For any probability measure $\mathbb{P}$ and its associated empirical measure $\mathbb{P}_n$, we have for any $\delta>0$,
$$\sup_{B\in\Xi_s}\left|\D_{\Xi_{2s}}(B,\mathbb{P}_n)-\D_{\Xi_{2s}}(B,\mathbb{P})\right|\leq C\sqrt{\frac{ms+s\log\left(\frac{ep}{s}\right)}{n}}+\sqrt{\frac{\log(1/\delta)}{2n}},$$
with probability at least $1-2\delta$, where $C>0$ is some absolute constant.
\end{proposition}
Proposition \ref{prop:group} is an extension of both Proposition \ref{prop:sparse} and Proposition \ref{prop:multiple}. The rate consists of two parts. The first part $\frac{ms}{n}$ is determined by the number of parameters. Since there are only $s$ nonzero rows of $B$, the number of parameters is $ms$. The second part $\frac{s\log\left(\frac{ep}{s}\right)}{n}$ is determined by the model selection complexity. Given the sparsity level $s$, there are ${p\choose s}$ possible models with different row supports. This contributes to the rate $n^{-1}\log{p\choose s}\asymp \frac{s\log\left(\frac{ep}{s}\right)}{n}$.

Define the quantity
\begin{equation}
\kappa=\inf_{|\supp(v)|=2s}\frac{\|\Sigma^{1/2}v\|}{\|v\|}.\label{eq:group-est}
\end{equation}
We are now ready to give the main result. Consider i.i.d. observations from $\mathbb{P}_{(\epsilon,B,Q)}=(1-\epsilon)P_B+\epsilon Q$.

\begin{thm}\label{thm:group}
Consider the estimator $\wh{B}$.
Assume that $\epsilon^2+\frac{ms+s\log\left(\frac{ep}{s}\right)}{n}$ is sufficiently small. Then, we have
\begin{eqnarray*}
\Tr((\wh{B}-B)^T\Sigma(\wh{B}-B)) &\leq& C\sigma^2\left(\frac{ms+s\log\left(\frac{ep}{s}\right)}{n}\vee\epsilon^2\right), \\
\fnorm{\wh{B}-B}^2 &\leq& C\frac{\sigma^2}{\kappa^2}\left(\frac{ms+s\log\left(\frac{ep}{s}\right)}{n}\vee\epsilon^2\right),
\end{eqnarray*}
with $\mathbb{P}_{(\epsilon,B,Q)}$-probability at least $1-\exp\left(-C'\left(ms+s\log\left(\frac{ep}{s}\right)+n\epsilon^2\right)\right)$ uniformly over all $Q$ and $B\in\Xi_s$, where $C,C'$ are some absolute constants.
\end{thm}

Theorem \ref{thm:group} is an extension of both Theorem \ref{thm:sparse} and Theorem \ref{thm:multiple}. When $m=1$, the problem is reduced to sparse linear regression, and $\wh{B}$ in (\ref{eq:group-est}) is the same as $\hat{\beta}$ in (\ref{eq:sparse-est}). Thus, the rates given by Theorem \ref{thm:group} recovers those of Theorem \ref{thm:sparse}. When $s=1$, this is the setting of multivariate linear regression without the group sparse structure, and the rates of Theorem \ref{thm:group} recover those of Theorem \ref{thm:multiple}.

The rates given by Theorem \ref{thm:group} are minimax optimal by Theorem \ref{thm:lower} and \cite{lounici2011oracle}.

\subsection{Reduced Rank Regression}

The final application is for reduced rank regression. In the multivariate linear regression setting 
 $Y=B^TX+\sigma Z$, the regression matrix $B$ is assumed to be low-rank. In particular, $B\in\mathcal{A}_r$, where $\mathcal{A}_r$ is defined in (\ref{eq:lrset}), except that the dimension $p_1\times p_2$ is replaced by $p\times m$.
 Some recent progresses on this topic were made by \cite{bunea2011optimal,ma2014adaptive} and references therein.
 
Define the estimator by
\begin{equation}
\wh{B}=\argmax_{B\in\mathcal{A}_r}\D_{\mathcal{A}_{2r}}(B,\{(X_i,Y_i)\}_{i=1}^n).\label{eq:rrr-est}
\end{equation}

We give the uniform convergence of the empirical depth function.
\begin{proposition}\label{prop:reduced}
For any probability measure $\mathbb{P}$ and its associated empirical measure $\mathbb{P}_n$, we have for any $\delta>0$,
$$\sup_{B\in\mathcal{A}_r}\left|\D_{\mathcal{A}_{2r}}(B,\mathbb{P}_n)-\D_{\mathcal{A}_{2r}}(B,\mathbb{P})\right|\leq C\sqrt{\frac{r(p+m)}{n}}+\sqrt{\frac{\log(1/\delta)}{2n}},$$
with probability at least $1-2\delta$, where $C>0$ is some absolute constant.
\end{proposition}
Note that Proposition \ref{prop:reduced} is an extension of Proposition \ref{prop:multiple}. For a full rank matrix, $r=p\wedge m$, and therefore $r(p+m)\asymp pm$.

To present the error rate of (\ref{eq:rrr-est}), 
define the quantity
$$\kappa=\inf_{v\neq 0}\frac{\|\Sigma^{1/2}v\|}{\|v\|}.$$
Consider i.i.d. observations from $\mathbb{P}_{(\epsilon,B,Q)}=(1-\epsilon)P_B+\epsilon Q$.

\begin{thm}\label{thm:rrr}
Consider the estimator $\wh{B}$.
Assume that $\epsilon^2+\frac{r(p+m)}{n}$ is sufficiently small. Then, we have
\begin{eqnarray*}
\Tr((\wh{B}-B)^T\Sigma(\wh{B}-B)) &\leq& C\sigma^2\left(\frac{r(p+m)}{n}\vee\epsilon^2\right),\\
\fnorm{\wh{B}-B}^2 &\leq& C\frac{\sigma^2}{\kappa^2}\left(\frac{r(p+m)}{n}\vee\epsilon^2\right),\\
\nunorm{\wh{B}-B}^2 &\leq& C\frac{\sigma^2}{\kappa^2}\left(\frac{r^2(p+m)}{n}\vee r\epsilon^2\right),
\end{eqnarray*}
with $\mathbb{P}_{(\epsilon,B,Q)}$-probability at least $1-\exp\left(-C'\left(r(p+m)+n\epsilon^2\right)\right)$ uniformly over all $Q$ and $B\in\mathcal{A}_r$, where $C,C'$ are some absolute constants.
\end{thm}

Theorem \ref{thm:rrr} gives rates in terms of prediction loss, squared Frobenius loss and squared nuclear loss. The rates are identical to those of Theorem \ref{thm:trace} for low-rank trace regression, with $p+m$ corresponding to $p_1+p_2$ in Theorem \ref{thm:trace}. This is due to the similarity of the two problems. In both problems, the regression matrix $B$ is assumed to belong to the low-rank set $\mathcal{A}_r$. The only difference is that for trace regression, the response is univariate and the covariate is a matrix, and for reduced rank regression, the response is multivariate and the covariate is a vector.

Applying the minimax lower bounds in \cite{ma2014adaptive}, we find that the rates given by Theorem \ref{thm:rrr} are optimal when $\epsilon=0$. Though the lower bounds in \cite{ma2014adaptive} are for a fixed design setting and they did not give explicit dependence on $\kappa$, the results can be easily modified to the random design setting considered here. The dependence on $\kappa$ can be made explicit as well. We refer the readers to the discussion in \cite{raskutti2011minimax,chen2015general} for details. In addition, Theorem \ref{thm:lower} and similar calculations of modulus of continuity as in Section \ref{sec:rank} imply that the rates given by Theorem \ref{thm:rrr} are also optimal when $\epsilon>0$.

\section{Discussion}\label{sec:disc}

\subsection{Extension to General Error Distributions}

The error distributions we consider in Section \ref{sec:uni} and Section \ref{sec:multi} are all Gaussian. This assumption can be greatly relaxed. In this section, we consider error distributions that have elliptical shapes.
\begin{definition} 
A random vector $W\sim EC(0,\Gamma,F)$ is distributed according to a centered continuous elliptical distribution with a scatter matrix $\Gamma\in\mathbb{R}^{d\times d}$ and a marginal cumulative distribution function $F$ if and only if $W=\Gamma^{1/2}E$, and $F(t)=\mathbb{P}(u^TE\leq t\|u\|)$ does not depend on $u\in\mathbb{R}^d$. Moreover, there is a density function $f$, such that $f(0)=1$ and $F(t)=\int_{-\infty}^tf(s)ds$.
\end{definition}

A more general definition of elliptical distributions is referred to \cite{fang1990symmetric}. Here, we only consider those 
that have marginal densities. Without loss of generality, we impose the constraint that $f(0)=1$. Otherwise, the scatter matrix $\Gamma$ would only be defined up to a multiplicative factor. When the dimension is $1$, the definition covers all random variables with symmetric density functions centered at zero.

Consider the setting of multivariate linear regression in Section \ref{sec:multiple}. The regression model $P_B$ for $X\in\mathbb{R}^p$ and $Y\in \mathbb{R}^m$ is specified by the sampling process $X\sim N(0,\Sigma)$ and $(Y-B^TX)|X\sim EC(0,\Gamma, F)$. 
The dependence on $\Sigma,\Gamma,F$ are suppressed in the notation of $P_B$. For i.i.d. observations generated by $\mathbb{P}_{(\epsilon,B,Q)}=(1-\epsilon)P_B+\epsilon Q$, the results of Theorem \ref{thm:multiple} are extended to the following theorem.

\begin{thm}\label{thm:EC}
Consider the estimator $\wh{B}$ defined in (\ref{eq:multiple-est}). Assume that $\epsilon^2+\frac{pm}{n}$ is sufficiently small. Moreover, there are some constants $c_1$ and $c_2$ such that $\min_{|t|\leq c_1}f(t)\geq c_2$. Then, we have
\begin{eqnarray}
\label{eq:multiple-pred}\Tr((\wh{B}-B)^T\Sigma(\wh{B}-B)) &\leq& C\sigma^2\left(\frac{pm}{n}\vee\epsilon^2\right), \\
\label{eq:multiple-frob}\fnorm{\wh{B}-B}^2 &\leq& C\frac{\sigma^2}{\kappa^2}\left(\frac{pm}{n}\vee\epsilon^2\right),
\end{eqnarray}
with $\mathbb{P}_{(\epsilon,B,Q)}$-probability at least $1-\exp\left(-C'\left(pm+n\epsilon^2\right)\right)$ uniformly over all $Q$ and $B\in\mathbb{R}^{p\times m}$, where $\kappa$ is defined in (\ref{eq:multiple-kappa}), $\sigma^2=\opnorm{\Gamma}$ and  $C,C'$ are some absolute constants.
\end{thm}

In addition to Theorem \ref{thm:multiple}, all the other results (except those of Gaussian graphical model) in Section \ref{sec:uni} and Section \ref{sec:multi} can be extended to the setting of general elliptical error distributions. The results are the same and thus the details are omitted. Theorem \ref{thm:EC} implies that the regression depth maximizer is not only robust to contamination, but is also robust to general error distributions.

Besides the error distribution, the Gaussian assumption for the covariates can also be extended similarly. However, this requires significantly more technical details and there are more than one ways to do it. We therefore do not explore all the possibilities here.

\subsection{A General Notion of Depth for Linear Operators}

Consider a general covariate space $\mathcal{X}$ and a general response space $\mathcal{Y}$. We assume the response space $\mathcal{Y}$ is a Hilbert space equipped with an inner product $\iprod{\cdot}{\cdot}$. Let $\ell(\mathcal{X},\mathcal{Y})$ be a class of linear operators $f:\mathcal{X}\rightarrow \mathcal{Y}$. The inner product structure on the response space allows us to define a general depth function for a linear operator $f\in\ell(\mathcal{X},\mathcal{Y})$. Given a probability distribution $(X,Y)\sim\mathbb{P}$ on $\mathcal{X}\times\mathcal{Y}$, the depth of an $f\in\ell(\mathcal{X},\mathcal{Y})$ is defined as
$$\D_{\mathcal{G}}(f,\mathbb{P})=\inf_{g\in\mathcal{G}}\mathbb{P}\left\{\iprod{g(X)}{Y-f(X)}\geq 0\right\},$$
where $\mathcal{G}$ is a subset of $\ell(\mathcal{X},\mathcal{Y})$.

This general definition not only covers the multivariate regression depth studied in this paper, but also allows the covariate to be a function. Some special cases are listed in the following table.
\begin{center}
  \begin{tabular}{ l | c | r }
    \hline
     & $\mathcal{X}$ & $\mathcal{Y}$ \\ \hline
    Tukey's depth \citep{tukey1975mathematics} & $\{1\}$ & $\mathbb{R}^m$ \\ \hline
    regression depth \citep{rousseeuw1999regression} & $\mathbb{R}^p$ & $\mathbb{R}$ \\ \hline
   multivariate regression depth \citep{bern2000multivariate,mizera2002depth} & $\mathbb{R}^p$ & $\mathbb{R}^m$ \\ \hline
   depth for functional linear regression & $\mathcal{C}[0,1]$ & $\mathbb{R}$ \\ \hline
   depth for multivariate functional linear regression & $\mathcal{C}[0,1]$ & $\mathbb{R}^m$ \\ \hline
  \end{tabular}
\end{center}
When $\mathcal{X}\times\mathcal{Y}$ takes $\{1\}\times \mathbb{R}^m$, $\mathbb{R}^p\times\mathbb{R}$ and $\mathbb{R}^p\times\mathbb{R}^m$, respectively, we recover Tukey's depth, regression depth and multivariate regression depth. Moreover, when $\mathcal{X}$ takes the class of all continuous functions on the unit interval $\mathcal{C}[0,1]$, the depth function can be used for robust functional linear regression. This application will be considered in future projects.

\section{Proofs}\label{sec:pf}

This section collects the proofs of the results presented in the paper. Section \ref{sec:pf-uniform} proves uniform convergence of all the empirical depth functions used in the paper. This includes the proofs of Propositions \ref{prop:nonpar}, \ref{prop:sparse}, \ref{prop:rank}, \ref{prop:multiple}, \ref{prop:group} and \ref{prop:reduced}. Section \ref{sec:pf-curve} establishes the curvature of the population depth functions. Finally, in Section \ref{sec:pf-main}, we prove all the theorems in the paper.

\subsection{Uniform Convergence of the Empirical Depth Functions}\label{sec:pf-uniform}

To establish uniform convergence of the empirical depth functions, it is essential to bound $\sup_{A\in\mathcal{A}}|\mathbb{P}_n(A)-\mathbb{P}(A)|$ over a collection $\mathcal{A}$. The first step is to use McDiarmid's bounded difference inequality. The following version can be found in Chapter 3.1 of \cite{devroye2012combinatorial}.
\begin{lemma}\label{lem:bdd}
For any probability measure $\mathbb{P}$ and its associated empirical measure $\mathbb{P}_n$, we have for any $t>0$,
$$\sup_{A\in\mathcal{A}}|\mathbb{P}_n(A)-\mathbb{P}(A)|\leq \mathbb{E}\left\{\sup_{A\in\mathcal{A}}|\mathbb{P}_n(A)-\mathbb{P}(A)|\right\}+t,$$
with probability at least $1-2e^{-2nt^2}$.
\end{lemma}

By Lemma \ref{lem:bdd}, it is sufficient to bound the expectation $\mathbb{E}\left\{\sup_{A\in\mathcal{A}}|\mathbb{P}_n(A)-\mathbb{P}(A)|\right\}$. This quantity can be controlled by the VC dimension of $\mathcal{A}$. The following lemma can be bound in Chapter 4.3 of \cite{devroye2012combinatorial}.

\begin{lemma}\label{lem:VC}
For any class $\mathcal{A}$ with VC dimension $V$,
$$\mathbb{E}\left\{\sup_{A\in\mathcal{A}}|\mathbb{P}_n(A)-\mathbb{P}(A)|\right\}\leq C\sqrt{\frac{V}{n}},$$
where $C>0$ is a universal constant.
\end{lemma}

Lemma \ref{lem:VC} suggests that we need to give an upper bound for the VC dimension of the class $\mathcal{A}$. For the depth functions considered in the paper, the relevant class is
\begin{equation}
\mathcal{A}=\left\{\{Z\in\mathbb{R}^{d_1\times d_2}:\Tr(WZ^T)\geq 0\}: W\in\mathbb{R}^{d_1\times d_2},\text{rank}(W)\leq r\right\}. \label{eq:low-rank-A}
\end{equation}
Intuitively, the matrix $W$ in $\mathcal{A}$ defined above has at most $r(d_1+d_2)$ degrees of freedom, which suggests a VC dimension bound $r(d_1+d_2)$. It was shown by \cite{wolf2007modeling} that the VC dimension of $\mathcal{A}$ is bounded by $r(d_1+d_2)\log(r(d_1+d_2))$. Using a slightly modified proof, we obtain a bound with the rate $r(d_1+d_2)$ at the cost of a larger constant.

\begin{lemma}\label{lem:VCrank}
The VC dimension of (\ref{eq:low-rank-A}) is bounded by $8r(d_1+d_2)$.
\end{lemma}
\begin{proof}
For a matrix with rank at most $r$, it has decomposition $W=\sum_{l=1}^ru_lv_l^T$. Thus, $\Tr(WZ^T)=\sum_{l=1}^ru_l^TZv_l=\sum_{l=1}^r\sum_{i=1}^{d_1}\sum_{j=1}^{d_2}Z_{ij}u_{li}v_{lj}$ is a polynomial of degree $2$ in $r(d_1+d_2)$ variables. According to \cite{warren1968lower,wolf2007modeling}, if there is some $x\geq r(d_1+d_2)$, such that
\begin{equation}
\left(\frac{8ex}{r(d_1+d_2)}\right)^{r(d_1+d_2)} \leq 2^x\label{eq:warren}
\end{equation}
holds, then the VC dimension is bounded by $x$. It is easy to check that $x=8r(d_1+d_2)$ satisfies (\ref{eq:warren}), and thus is an upper bound for the VC dimension.
\end{proof}

Now we are ready to give proofs for all the uniform convergence results of the empirical depth functions.

\begin{proof}[Proof of Proposition \ref{prop:nonpar}]
For a general multi-task regression depth function, we have
\begin{eqnarray*}
&& \sup_{B\in\mathcal{B}}|\D_{\U}(B,\mathbb{P})-\D_{\U}(B,\mathbb{P}_n)| \\
&\leq& \sup_{B\in\mathcal{B}}\sup_{U\in\U}\left|\mathbb{P}_n\left\{\iprod{U^TX}{Y-B^TX}\geq 0\right\}-\mathbb{P}\left\{\iprod{U^TX}{Y-B^TX}\geq 0\right\}\right|.
\end{eqnarray*}
Since
\begin{eqnarray*}
&& \iprod{U^TX}{Y-B^TX} \\
&=& \Tr(UYX^T)-\Tr(UB^TXX^T) \\
&=& \Tr(WZ^T),
\end{eqnarray*}
where
\begin{equation}
W=W(U,B)=\begin{pmatrix}
U & 0 \\
0 & UB^T
\end{pmatrix}\quad\text{and}\quad Z^T=\begin{pmatrix}
YX^T & 0 \\
0 & -XX^T
\end{pmatrix},\label{eq:matrix-W}
\end{equation}
we have
\begin{equation}
\sup_{B\in\mathcal{B}}|\D_{\U}(B,\mathbb{P})-\D_{\U}(B,\mathbb{P}_n)|\leq \sup_{A\in\mathcal{A}}|\mathbb{P}_n(A)-\mathbb{P}(A)|.\label{eq:uniform}
\end{equation}
We use $\mathbb{P}$ to denote the distribution of $Z$ with slight abuse of notation. The set $\mathcal{A}$ is defined as
\begin{equation}
\mathcal{A}=\left\{\{Z\in\mathbb{R}^{2p\times (p+m)}:\Tr(WZ^T)\geq 0\}: W=W(U,B),U\in\U,B\in\mathcal{B}\right\}.\label{eq:A}
\end{equation}
In the setting of Proposition \ref{prop:nonpar},
\begin{equation}
W=\begin{pmatrix}
u & 0 \\
0 & u\beta^T
\end{pmatrix},\label{eq:vector-W}
\end{equation}
for any $u\in\mathbb{R}^{k}$ and $\beta\in\mathbb{R}^k$. Hence, $W$ is of rank at most $1$. By Lemma \ref{lem:bdd}, Lemma \ref{lem:VC} and Lemma \ref{lem:VCrank}, we obtain the desired result.
\end{proof}

\begin{proof}[Proof of Proposition \ref{prop:sparse}]
The same argument that leads to (\ref{eq:uniform}) gives the bound
$$\sup_{\beta\in\Theta_s}\left|\D_{\Theta_{2s}}(\beta,\mathbb{P}_n)-\D_{\Theta_{2s}}(\beta,\mathbb{P})\right|\leq \max_{S_1\in\{S\subset[p]:|S|=s\},S_2\in\{S\subset[p]:|S|=2s\}}\sup_{A\in\mathcal{A}_{S_1,S_2}}|\mathbb{P}_n(A)-\mathbb{P}(A)|,$$
where
$$\mathcal{A}_{S_1,S_2}=\left\{\{Z\in\mathbb{R}^{2p\times (p+1)}:\Tr(WZ^T)\geq 0\}: W=W(u,\beta),u\in\Theta_{S_2},\beta\in\Theta_{S_1}\right\},$$
and $W(u,\beta)$ is in the form of (\ref{eq:vector-W}).
For any subset $S\subset[p]$, $\Theta_S$ is defined as
$$\Theta_S=\left\{u\in\mathbb{R}^p: u_j=0\text{ for all }j\in S^c\right\}.$$
By Lemma \ref{lem:bdd} and a union bound argument, we have
\begin{eqnarray*}
&& \sup_{\beta\in\Theta_s}\left|\D_{\Theta_{2s}}(\beta,\mathbb{P}_n)-\D_{\Theta_{2s}}(\beta,\mathbb{P})\right| \\
&& \leq \max_{S_1\in\{S\subset[p]:|S|=s\},S_2\in\{S\subset[p]:|S|=2s\}}\mathbb{E}\left\{\sup_{A\in\mathcal{A}_{S_1,S_2}}|\mathbb{P}_n(A)-\mathbb{P}(A)|\right\}+t,
\end{eqnarray*}
with probability at least $1-2e^{-2nt^2+4s\log\left(\frac{ep}{s}\right)}$. Finally, in view of Lemma \ref{lem:VC}, it is sufficient to upper bound the VC dimension of $\mathcal{A}_{S_1,S_2}$. Note that $\mathcal{A}_{S_1,S_2}$ contains matrices of the form (\ref{eq:vector-W}) with $u\in\Theta_{S_2}$ and $\beta\in\Theta_{S_1}$, the VC dimension is bounded by $8(5s+1)$ according to Lemma \ref{lem:VCrank}. Hence, we have
$$\sup_{\beta\in\Theta_s}\left|\D_{\Theta_{2s}}(\beta,\mathbb{P}_n)-\D_{\Theta_{2s}}(\beta,\mathbb{P})\right|\leq C\sqrt{\frac{s}{n}}+t,$$
with probability at least $1-2e^{-2nt^2+4s\log\left(\frac{ep}{s}\right)}$ for some universal constant $C>0$. The desired result follows by setting $t^2=\frac{4s\log\left(\frac{ep}{s}\right)+\log(1/\delta)}{2n}$.
\end{proof}

\begin{proof}[Proof of Proposition \ref{prop:rank}]
For the trace regression depth function, we have
\begin{eqnarray*}
&& \sup_{B\in\mathcal{A}_r}\left|\D_{\mathcal{A}_{2r}}(B,\mathbb{P}_n)-\D_{\mathcal{A}_{2r}}(B,\mathbb{P})\right| \\
&\leq& \sup_{B\in\mathcal{A}_r}\sup_{U\in\mathcal{A}_{2r}}|\mathbb{P}_n\left\{\iprod{U}{X}(y-\iprod{B}{X})\geq 0\right\}-\mathbb{P}\left\{\iprod{U}{X}(y-\iprod{B}{X})\geq 0\right\}|.
\end{eqnarray*}
Since
\begin{eqnarray*}
&& \iprod{U}{X}(y-\iprod{B}{X}) \\
&=& U^TXy-\Tr(BU^TXX^T) \\
&=& \Tr(WZ^T),
\end{eqnarray*}
where
$$W=W(U,B)=\begin{pmatrix}
U^T & 0 \\
0 & BU^T
\end{pmatrix}\quad\text{and}\quad Z^T=\begin{pmatrix}
Xy & 0 \\
0 & -XX^T
\end{pmatrix},$$
and
we thus have
$$\sup_{B\in\mathcal{A}_r}\left|\D_{\mathcal{A}_{2r}}(B,\mathbb{P}_n)-\D_{\mathcal{A}_{2r}}(B,\mathbb{P})\right|\leq \sup_{A\in\mathcal{A}}|\mathbb{P}_n(A)-\mathbb{P}(A)|.$$
We use $\mathbb{P}$ to denote the distribution of $Z$ with slight abuse of notation. The set $\mathcal{A}$ is defined as
$$\mathcal{A}=\left\{\{Z\in\mathbb{R}^{(p_1+p_2)\times 2p_1}:\Tr(WZ^T)\geq 0\}: W=W(U,B),U\in\U,B\in\mathcal{B}\right\}.$$
By Lemma \ref{lem:VCrank}, the VC dimension of $\mathcal{A}$ is bounded by $16r(3p_1+p_2)$. Together with Lemma \ref{lem:bdd} and Lemma \ref{lem:VC}, we obtain the desired result.
\end{proof}

\begin{proof}[Proof of Proposition \ref{prop:multiple}]
Using the argument that leads to (\ref{eq:uniform}), we have
$$\sup_{B\in\mathbb{R}^{p\times m}}|\D_{\mathbb{R}^{p\times m}\backslash\{0\}}(B,\mathbb{P})-\D_{\mathbb{R}^{p\times m}\backslash\{0\}}(B,\mathbb{P}_n)|\leq \sup_{A\in\mathcal{A}}|\mathbb{P}_n(A)-\mathbb{P}(A)|,$$
where $\mathcal{A}$ is defined in (\ref{eq:A}), which involves matrices $W$ of dimension $2p\times (p+m)$ with rank at most $p\wedge m$. According to Lemma \ref{lem:VCrank}, its VC dimension is bounded by $8(p\wedge m)(3p+m)\leq 32pm$. Together with Lemma \ref{lem:bdd} and Lemma \ref{lem:VC}, we obtain the desired result.
\end{proof}

\begin{proof}[Proof of Proposition \ref{prop:group}]
The same argument that leads to (\ref{eq:uniform}) gives the bound
$$\sup_{B\in\Xi_s}\left|\D_{\Xi_{2s}}(B,\mathbb{P}_n)-\D_{\Xi_{2s}}(B,\mathbb{P})\right|\leq \max_{S_1\in\{S\subset[p]:|S|=s\},S_2\in\{S\subset[p]:|S|=2s\}}\sup_{A\in\mathcal{A}_{S_1,S_2}}|\mathbb{P}_n(A)-\mathbb{P}(A)|,$$
where
$$\mathcal{A}_{S_1,S_2}=\left\{\{Z\in\mathbb{R}^{2p\times (p+m)}:\Tr(WZ^T)\geq 0\}: W=W(U,B),U\in\Xi_{S_2},B\in\Xi_{S_1}\right\},$$
and $W(U,B)$ is defined in (\ref{eq:matrix-W}).
For any subset $S\subset[p]$, $\Xi_S$ is defined as
$$\Xi_S=\left\{U\in\mathbb{R}^{p\times m}: U_{j*}=0\text{ for all }j\in S^c\right\}.$$
By Lemma \ref{lem:bdd} and a union bound argument, we have
\begin{eqnarray*}
&& \sup_{B\in\Xi_s}\left|\D_{\Xi_{2s}}(B,\mathbb{P}_n)-\D_{\Xi_{2s}}(B,\mathbb{P})\right| \\
&& \leq \max_{S_1\in\{S\subset[p]:|S|=s\},S_2\in\{S\subset[p]:|S|=2s\}}\mathbb{E}\left\{\sup_{A\in\mathcal{A}_{S_1,S_2}}|\mathbb{P}_n(A)-\mathbb{P}(A)|\right\}+t,
\end{eqnarray*}
with probability at least $1-2e^{-2nt^2+4s\log\left(\frac{ep}{s}\right)}$. Finally, in view of Lemma \ref{lem:VC}, it is sufficient to upper bound the VC dimension of $\mathcal{A}_{S_1,S_2}$. Note that $\mathcal{A}_{S_1,S_2}$ contains matrices of the form (\ref{eq:matrix-W}) with $U\in\Xi_{S_2}$ and $B\in\Xi_{S_1}$, the VC dimension is bounded by $8(2s\wedge m)(5s+m)\leq 64ms$ according to Lemma \ref{lem:VCrank}. Hence, we have
$$ \sup_{B\in\Xi_s}\left|\D_{\Xi_{2s}}(B,\mathbb{P}_n)-\D_{\Xi_{2s}}(B,\mathbb{P})\right|\leq C\sqrt{\frac{ms}{n}}+t,$$
with probability at least $1-2e^{-2nt^2+4s\log\left(\frac{ep}{s}\right)}$ for some universal constant $C>0$. The desired result follows by setting $t^2=\frac{4s\log\left(\frac{ep}{s}\right)+\log(1/\delta)}{2n}$.
\end{proof}

\begin{proof}[Proof of Proposition \ref{prop:reduced}]
Using the argument that leads to (\ref{eq:uniform}), we have
$$\sup_{B\in\mathcal{A}_r}\left|\D_{\mathcal{A}_{2r}}(B,\mathbb{P}_n)-\D_{\mathcal{A}_{2r}}(B,\mathbb{P})\right|\leq \sup_{A\in\mathcal{A}}|\mathbb{P}_n(A)-\mathbb{P}(A)|,$$
where $\mathcal{A}$ is defined in (\ref{eq:A}), which involves matrices $W$ of dimension $2p\times (p+m)$ with rank at most $2r$. According to Lemma \ref{lem:VCrank}, its VC dimension is bounded by $16r(3p+m)$. Together with Lemma \ref{lem:bdd} and Lemma \ref{lem:VC}, we obtain the desired result.
\end{proof}

\subsection{Curvature of the Populational Depth Functions}\label{sec:pf-curve}

In addition to the uniform convergence results, another key ingredient we need is the curvature of the population depth function. They are characterized for both univariate regression and multivariate regression by the following two lemmas, respectively.

\begin{lemma}\label{lem:curve}
Let $P_{\beta}$ denote the joint distribution of $(X,y)\in\mathbb{R}^p\times\mathbb{R}$ specified by $X\sim N(0,\Sigma)$ and $y|X\sim N(\beta^TX,\sigma^2)$. For any $\alpha\in\mathbb{R}^p$ such that $\alpha-\beta\in\U$, as long as $\D_{\U}(\alpha,P_{\beta})\geq\frac{1}{2}-\eta$ for some $\eta<\frac{5}{12}$, we have
$$\|\Sigma^{1/2}(\alpha-\beta)\|\leq C\sigma\eta,$$
where $C>0$ is  some universal constant.
\end{lemma}
\begin{proof}
By the definition of the depth function, we have
$$\D_{\U}(\alpha,P_{\beta})=1-\sup_{u\in\U}\mathbb{E}\Phi\left(\frac{u^TXX^T(\alpha-\beta)}{\sigma|u^TX|}\right),$$
where $\Phi(\cdot)$ is the cumulative distribution function of $N(0,1)$.
Together with the condition $\D_{\U}(\alpha,P_{\beta})\geq\frac{1}{2}-\eta$, we obtain
$$\sup_{u\in\U}\mathbb{E}\Phi\left(\frac{u^TXX^T(\alpha-\beta)}{\sigma|u^TX|}\right)-\Phi(0)\leq \eta.$$
Since $\alpha-\beta\in\U$, we have
$$\mathbb{E}\Phi\left(\frac{|X^T(\alpha-\beta)|}{\sigma}\right)-\Phi(0)\leq \eta.$$
For $Z\sim N(0,1)$, consider the function $g(t)=\mathbb{E}\Phi(t|Z|)$. It is easy to check that $g(t)$ is increasing in $t$. Since $g(4)>11/12$, the fact that $g(t)\leq 1/2+\eta$ for some $\eta<5/12$ implies that $t\leq 4$. The definition of $g(t)$ implies that
$$g(t)-\frac{1}{2}=\mathbb{E}\Phi(t|Z|)-\Phi(0)\geq \phi(t)\mathbb{E}\min\{t,t|Z|\}\geq t\phi(4)\mathbb{E}\min\{1,|Z|\},$$
where $\phi(\cdot)$ is the density function of $N(0,1)$.
Therefore,
$$\mathbb{E}\Phi\left(\frac{|X^T(\alpha-\beta)|}{\sigma}\right)-\Phi(0)=g\left(\frac{\|\Sigma^{1/2}(\alpha-\beta)\|}{\sigma}\right)-\frac{1}{2}\geq c \frac{\|\Sigma^{1/2}(\alpha-\beta)\|}{\sigma},$$
where $c=\phi(4)\mathbb{E}\min\{1,|Z|\}$. This leads to the conclusion
$$\|\Sigma^{1/2}(\alpha-\beta)\|\leq c^{-1}\sigma\eta.$$
Thus, the proof is complete.
\end{proof}

\begin{lemma}\label{lem:curve-m}
Let $P_{B}$ denote the joint distribution of $(X,Y)\in\mathbb{R}^p\times\mathbb{R}^m$ specified by $X\sim N(0,\Sigma)$ and $Y|X\sim N(B^TX,\sigma^2I_m)$. For any $A\in\mathbb{R}^{p\times m}$ such that $A-B\in\U$, as long as $\D_{\U}(A,P_B)\geq \frac{1}{2}-\eta$ for some $\eta<\frac{3}{20}$, we have
$$\sqrt{\Tr((A-B)^T\Sigma(A-B))}\leq C\sigma\eta,$$
where $C>0$ is  some universal constant.
\end{lemma}
\begin{proof}
By the definition of the depth function, we have
$$\D_{\U}(A,P_B)=1-\sup_{U\in\U}\mathbb{E}\Phi\left(\sigma^{-1}\iprod{\frac{U^TX}{\|U^TX\|}}{(A-B)^TX}\right),$$
where $\Phi(\cdot)$ is the cumulative distribution function of $N(0,1)$. Together with the condition $\D_{\U}(A,P_B)\geq \frac{1}{2}-\eta$, we obtain
$$\sup_{U\in\U}\mathbb{E}\Phi\left(\sigma^{-1}\iprod{\frac{U^TX}{\|U^TX\|}}{(A-B)^TX}\right)-\Phi(0)\leq\eta.$$
Sime $A-B\in\U$, we have
\begin{equation}
\mathbb{E}\Phi\left(\frac{\|(A-B)^TX\|}{\sigma}\right)-\Phi(0)\leq\eta.\label{eq:depth-bound}
\end{equation}
Consider the random variable $Y=\frac{\|(A-B)^TX\|^2}{\Tr((A-B)^T\Sigma(A-B))}$. We need a lower bound for the probability $\mathbb{P}(Y>c)$. By Cauchy-Schwarz inequality, we have
$$\mathbb{E}Y\leq c+\mathbb{E}Y\mathbb{I}\{Y>c\}\leq c+\sqrt{\mathbb{E}Y^2}\sqrt{\mathbb{P}(Y>c)}.$$
This leads to the inequality
\begin{equation}
\sqrt{\mathbb{P}(Y>c)} \geq \frac{\mathbb{E}Y-c}{\sqrt{\mathbb{E}Y^2}}.\label{eq:lower-tail}
\end{equation}
Thus, we need a lower bound for $\mathbb{E}Y$ and an upper bound for $\mathbb{E}Y^2$. It is easy to see that $\mathbb{E}Y=1$. To bound $\mathbb{E}Y^2$, we write
$$\|(A-B)^TX\|^2=\|(A-B)^T\Sigma^{1/2}Z\|^2=\sum_{j=1}^m|K_j^TZ|^2,$$
where $Z\sim N(0,I_p)$. Thus,
$$\Tr((A-B)^T\Sigma(A-B))=\sum_{j=1}^m\|K_j\|^2.$$
Therefore,
\begin{eqnarray*}
\mathbb{E}\|(A-B)^TX\|^4 &=& \sum_{j=1}^m\sum_{l=1}^m\mathbb{E}|K_j^TZ|^2|K_l^TZ|^2 \\
&=& \sum_{j=1}^m\sum_{l=1}^m\|K_j\|^2\|K_l\|^2\mathbb{E}\frac{|K_j^TZ|^2}{\|K_j\|^2}\frac{|K_l^TZ|^2}{\|K_l\|^2} \\
&\leq& 3\sum_{j=1}^m\sum_{l=1}^m\|K_j\|^2\|K_l\|^2 \\
&=& 3\left(\sum_{j=1}^m\|K_j\|^2\right)^2.
\end{eqnarray*}
Hence,
$$\mathbb{E}Y^2= \frac{\mathbb{E}\|(A-B)^TX\|^4}{\left(\sum_{j=1}^m\|K_j\|^2\right)^2}\leq 3.$$
The inequality (\ref{eq:lower-tail}) leads to
$$\mathbb{P}\left(Y>\frac{1}{4}\right) \geq \frac{3}{16}.$$
Now we define the function
$$g(t)=\mathbb{E}\Phi\left(t\frac{\|(A-B)^TX\|}{\sqrt{\Tr((A-B)^T\Sigma(A-B))}}\right)=\mathbb{E}\Phi(t\sqrt{Y}).$$
It is easy to check that $g(t)$ is increasing in $t$. Moreover,
\begin{eqnarray*}
g(4) &=& \mathbb{E}\Phi(4\sqrt{Y}) \\
&=& \mathbb{E}\Phi(4\sqrt{Y})\mathbb{I}\left\{Y>\frac{1}{4}\right\}+\mathbb{E}\Phi(4\sqrt{Y})\mathbb{I}\left\{Y\leq\frac{1}{4}\right\} \\
&\geq& \Phi(2)\mathbb{P}\left(Y>\frac{1}{4}\right)+\Phi(0)\mathbb{P}\left(Y\leq\frac{1}{4}\right) \\
&=& \left(\Phi(2)-\Phi(0)\right)\mathbb{P}\left(Y>\frac{1}{4}\right)+\frac{1}{2} \\
&\geq& \frac{1}{2}+\frac{3}{20}.
\end{eqnarray*}
Therefore, $g(t)\leq\frac{1}{2}+\eta$ for some $\eta<\frac{3}{20}$ implies that $t\leq 4$. The definition of $g(t)$ implies that
$$g(t)-\frac{1}{2}=\mathbb{E}\Phi(t\sqrt{Y})-\Phi(0)\geq t\phi(t)\mathbb{E}\min\{1,\sqrt{Y}\}\geq t\phi(4)\mathbb{E}\min\{1,\sqrt{Y}\},$$
where $\phi(\cdot)$ is the density function of $N(0,1)$.
Finally, we need to lower bound $\mathbb{E}\min\{1,\sqrt{Y}\}$. We have
\begin{eqnarray*}
\mathbb{E}\min\{1,\sqrt{Y}\} &\geq& \frac{1}{2}\mathbb{P}\left(\min\{1,\sqrt{Y}\}>\frac{1}{2}\right) \\
&\geq& \frac{1}{2}\mathbb{P}\left(Y>\frac{1}{4}\right) \\
&\geq& \frac{3}{32}.
\end{eqnarray*}
Hence,
\begin{eqnarray*}
\mathbb{E}\Phi\left(\frac{\|(A-B)^TX\|}{\sigma}\right)-\Phi(0) &=& g\left(\frac{\sqrt{\Tr((A-B)^T\Sigma(A-B))}}{\sigma}\right)-\frac{1}{2} \\
&\geq& c\frac{\sqrt{\Tr((A-B)^T\Sigma(A-B))}}{\sigma},
\end{eqnarray*}
where $c=\frac{3\phi(4)}{32}$. Using (\ref{eq:depth-bound}), we obtain the desired conclusion, and the proof is complete.
\end{proof}

\subsection{Proofs of Main Results}\label{sec:pf-main}

This section gives proofs of Theorems \ref{thm:nonpar}, \ref{thm:sparse}, \ref{thm:GGM}, \ref{thm:trace}, \ref{thm:multiple}, \ref{thm:group} and \ref{thm:rrr} as well as Theorem \ref{thm:EC}.
For i.i.d. data $\{(X_i,Y_i)\}_{i=1}^n$ from a contaminated distribution $(1-\epsilon)P+\epsilon Q$, it can be written as $\{(X_i^P,Y_i^P)\}_{i=1}^{n_1}\cup\{(X_i^Q,Y_i^Q)\}_{i=1}^{n_2}$. Marginally, we have $n_2\sim\text{Binomial}(n,\epsilon)$ and $n_1=n-n_2$. Conditioning on $n_1$ and $n_2$, $\{(X_i^P,Y_i^P)\}_{i=1}^{n_1}$ are i.i.d. from $P$ and $\{(X_i^Q,Y_i^Q)\}_{i=1}^{n_2}$ are i.i.d. from $Q$. The following lemma controls the ratio $n_2/n_1$.
\begin{lemma} \label{lem:ratio}
Assume $\epsilon<1/2$.
For any $\delta>0$ satisfying $n^{-1}\log(1/\delta)<c$ for some sufficiently small constant $c$, we have
\begin{equation}
\frac{n_2}{n_1}\leq \frac{\epsilon}{1-\epsilon}+C\sqrt{\frac{\log(1/\delta)}{n}},\label{eq:ratio}
\end{equation}
with probability at least $1-\delta$, where $C>0$ is a universal constant.
\end{lemma}

Now we are ready to prove the main results.

\begin{proof}[Proof of Theorem \ref{thm:nonpar}]
By Lemma \ref{lem:ratio}, we decompose the data $\{(X_i,y_i)\}_{i=1}^n=\{(X_i^P,y_i^P)\}_{i=1}^{n_1}\cup\{(X_i^Q,y_i^Q)\}_{i=1}^{n_2}$. The following analysis is on the intersection of the events of (\ref{eq:ratio}) and Proposition \ref{prop:nonpar} that holds with probability at least $1-2\delta$.
For any $f=\sum_{j=1}^{\infty}\beta_j\phi_j\in S_{\alpha}(M)$, there exists some $\beta_{[k]}\in\U_k$, such that for the corresponding $f_{[k]}$,
\begin{equation}
\|f_{[k]}-f\|^2=\|\beta_{[k]}-\beta\|^2\leq C_1k^{-2\alpha},\label{eq:bias}
\end{equation}
for some constant $C_1>0$ that only depends on $\alpha$ and $M$. Recall the notation $P_f$. By the definition of the depth function and Proposition \ref{prop:nonpar}, we have
\begin{eqnarray}
\label{eq:dU1}\D_{\U_k}(\hat{\beta},P_{f}) &\geq& \D_{\U_k}(\hat{\beta},\{(X_i^P,y_i^P)\}_{i=1}^{n_1})-C\sqrt{\frac{k}{n_1}}-\sqrt{\frac{\log(1/\delta)}{2n_1}} \\
\label{eq:bf1}&\geq& \frac{n}{n_1}\D_{\U_k}(\hat{\beta},\{(X_i,y_i)\}_{i=1}^{n})-\frac{n_2}{n_1}-C\sqrt{\frac{k}{n_1}}-\sqrt{\frac{\log(1/\delta)}{2n_1}} \\
\label{eq:defGamma}&\geq& \frac{n}{n_1}\D_{\U_k}(\beta_{[k]},\{(X_i,y_i)\}_{i=1}^{n})-\frac{n_2}{n_1}-C\sqrt{\frac{k}{n_1}}-\sqrt{\frac{\log(1/\delta)}{2n_1}} \\
\label{eq:bf2}&\geq& \D_{\U_k}(\beta_{[k]},\{(X_i^P,y_i^P)\}_{i=1}^{n_1})-\frac{n_2}{n_1}-C\sqrt{\frac{k}{n_1}}-\sqrt{\frac{\log(1/\delta)}{2n_1}} \\
\label{eq:dU2}&\geq& \D_{\U_k}(\beta_{[k]},P_f)-\frac{n_2}{n_1}-2C\sqrt{\frac{k}{n_1}}-2\sqrt{\frac{\log(1/\delta)}{2n_1}}.
\end{eqnarray}
The inequalities (\ref{eq:dU1}) and (\ref{eq:dU2}) are by Proposition \ref{prop:nonpar}. The inequalities (\ref{eq:bf1}) and (\ref{eq:bf2}) are due to the property of depth function that
$$n_1\mathcal{D}_{\mathcal{U}_k}(\beta,\{Y_i\}_{i=1}^{n_1})\geq n\mathcal{D}_{\mathcal{U}_k}(\beta,\{X_i\}_{i=1}^{n})-n_2\geq n_1\mathcal{D}_{\mathcal{U}_k}(\beta,\{Y_i\}_{i=1}^{n_1})-n_2,$$
for any $\beta\in\U_k$. The inequality (\ref{eq:defGamma}) is by the definition of $\hat{\beta}$. Moreover,
\begin{eqnarray}
\nonumber && \left|\D_{\U_k}(\beta_{[k]},P_f)-\D_{\U_k}(\beta,P_f)\right| \\
\nonumber &\leq& \sup_{u\in\U_k}\left|P_f\left(u^TX(y-X^T\beta)\geq 0\right)-P_f\left(u^TX(y-X^T\beta_{[k]})\geq 0\right)\right| \\
\label{eq:usePhi}&=& \sup_{u\in\U_k}\left|\mathbb{E}\Phi\left(\frac{u^TXX^T(\beta_{[k]}-\beta)}{|u^TX|}\right)-\Phi(0)\right| \\
\nonumber &\leq& \sqrt{\frac{1}{2\pi}}\mathbb{E}\left|X^T(\beta_{[k]}-\beta)\right| \\
\label{eq:useunif}&\leq& \sqrt{\frac{1}{2\pi}}\sqrt{\mathbb{E}(f_{[k]}(x)-f(x))^2} \\
\nonumber &=& \sqrt{\frac{1}{2\pi}}\|f_{[k]}-f\| \\
\label{eq:usesob}&\leq& C_1^{1/2}\sqrt{\frac{1}{2\pi}}k^{-\alpha},
\end{eqnarray}
where $\Phi(\cdot)$ is the cumulative distribution function of $N(0,1)$ in (\ref{eq:usePhi}) and $x\sim\text{Unif}[0,1]$ in (\ref{eq:useunif}). The inequality (\ref{eq:usesob}) is due to (\ref{eq:bias}). Therefore,
$$\D_{\U_k}(\beta_{[k]},P_f)\geq \frac{1}{2}-C_1^{1/2}\sqrt{\frac{1}{2\pi}}k^{-\alpha}.$$
Together with the inequality (\ref{eq:dU2}) and Lemma \ref{lem:ratio}, we have
\begin{equation}
\D_{\U_k}(\hat{\beta},P_{f})\geq \frac{1}{2}-\frac{\epsilon}{1-\epsilon}-C_2\left(\sqrt{\frac{k}{n}}+k^{-\alpha}+\sqrt{\frac{\log(1/\delta)}{n}}\right),\label{eq:startingf}
\end{equation}
with probability at least $1-2\delta$. At this point, we cannot directly use Lemma \ref{lem:curve}, because $\hat{\beta}-\beta\notin\U_k$. A slightly different argument is needed. Starting from (\ref{eq:startingf}), we have
$$\sup_{u\in\U_k}\mathbb{E}\Phi\left(\frac{u^TXX^T(\hat{\beta}-\beta)}{|u^TX|}\right)-\Phi(0)\leq \frac{\epsilon}{1-\epsilon}+C_2\left(\sqrt{\frac{k}{n}}+k^{-\alpha}+\sqrt{\frac{\log(1/\delta)}{n}}\right),$$
where the expectation is only taken over $X$.
The same argument that leads to (\ref{eq:usesob}) gives
$$\sup_{u\in\U_k}\mathbb{E}\Phi\left(\frac{u^TXX^T(\hat{\beta}-\beta_{[k]})}{|u^TX|}\right)-\Phi(0)\leq \frac{\epsilon}{1-\epsilon}+C_3\left(\sqrt{\frac{k}{n}}+k^{-\alpha}+\sqrt{\frac{\log(1/\delta)}{n}}\right).$$
Now, since $\hat{\beta}-\beta_{[k]}\in\U_k$, by the same argument in the proof of Lemma \ref{lem:curve}, we have
$$\|\hat{f}-f_{[k]}\|\leq C_4\left(\epsilon+\sqrt{\frac{k}{n}}+k^{-\alpha}+\sqrt{\frac{\log(1/\delta)}{n}}\right).$$
Using (\ref{eq:bias}) again, we have
$$\|\hat{f}-f\|\leq C_5\left(\epsilon+\sqrt{\frac{k}{n}}+k^{-\alpha}+\sqrt{\frac{\log(1/\delta)}{n}}\right).$$
The choice $k=\ceil{n^{\frac{1}{2\alpha+1}}}$ completes the proof.
\end{proof}

\begin{proof}[Proofs of Theorem \ref{thm:sparse} and Theorem \ref{thm:trace}]
We first give the proof of Theorem \ref{thm:sparse}.
By Lemma \ref{lem:ratio}, we decompose the data $\{(X_i,y_i)\}_{i=1}^n=\{(X_i^P,y_i^P)\}_{i=1}^{n_1}\cup\{(X_i^Q,y_i^Q)\}_{i=1}^{n_2}$. The following analysis is on the intersection of the events of (\ref{eq:ratio}) and Proposition \ref{prop:sparse} that holds with probability at least $1-2\delta$. Recall the notation $P_{\beta}$. Using the same arguments in (\ref{eq:dU1})--(\ref{eq:dU2}), we get
$$\D_{\Theta_{2s}}(\hat{\beta},P_{\beta})\geq \frac{1}{2}-\frac{n_2}{n_1}-2C\sqrt{\frac{s\log\left(\frac{ep}{s}\right)}{n_1}}-2\sqrt{\frac{\log(1/\delta)}{2n_1}}.$$
Lemma \ref{lem:ratio} implies that
\begin{equation}
\D_{\Theta_{2s}}(\hat{\beta},P_{\beta})\geq \frac{1}{2}-\frac{\epsilon}{1-\epsilon}-C_1\left(\sqrt{\frac{s\log\left(\frac{ep}{s}\right)}{n}}+\sqrt{\frac{\log(1/\delta)}{2n}}\right),\label{eq:interface}
\end{equation}
with probability at least $1-2\delta$. Since $\hat{\beta}-\beta\in\Theta_{2s}$, we use Lemma \ref{lem:curve} to deduce (\ref{eq:sparse-pred}). The bounds (\ref{eq:sparse-l2}) and (\ref{eq:sparse-l1}) are direct implications of (\ref{eq:sparse-pred}) by the definition of $\kappa$. Thus, the proof of Theorem \ref{thm:sparse} is complete. The proof of Theorem \ref{thm:trace} follows the same argument, and we do not repeat the details.
\end{proof}

\begin{proof}[Proof of Theorem \ref{thm:GGM}]
We use $\D_1$ to denote the first half of the data and $\D_2$ to denote the second half.
The model $X_j=\beta_{(j)}^TX_{-j}+\xi_j$ is an instance of sparse linear regression in Section \ref{sec:sparse}. Thus, the result of Theorem \ref{thm:sparse} implies that
$$\|\Sigma_{-j,-j}^{1/2}(\beta_{(j)}-\hat{\beta}_{(j)})\|^2\leq C\left(\frac{s\log\left(\frac{ep}{s}\right)}{n}\vee \epsilon^2+\frac{\log(1/\delta)}{n}\right),$$
and
$$\|\hat{\beta}_{(j)}-\beta_{(j)}\|^2_1\leq C\left(\frac{s^2\log\left(\frac{ep}{s}\right)}{n}\vee s\epsilon^2+\frac{s\log(1/\delta)}{n}\right),$$
with probability at least $1-2\delta$. The matrix $\Sigma_{-j,-j}$ is the covariance of $X_{-j}$. Now we study the error of $\wh{\Omega}_{jj}^{-1}$. Conditioning on $\D_1$,
$$X_j-\hat{\beta}_{(j)}^TX_{-j}=(\beta_{(j)}-\hat{\beta}_{(j)})^TX_{-j}+\xi_j\sim (1-\epsilon)N(0,\|\Sigma_{-j,-j}^{1/2}(\beta_{(j)}-\hat{\beta}_{(j)})\|^2+\Omega_{jj}^{-1})+\epsilon Q.$$
Theorem 3.1 of \cite{chen2015robust} implies that
$$|\wh{\Omega}_{jj}^{-1}-\Omega_{jj}^{-1}|^2\leq 2\|\Sigma_{-j,-j}^{1/2}(\beta_{(j)}-\hat{\beta}_{(j)})\|^4+C_1\left(\epsilon^2+\frac{\log(1/\delta)}{n}\right),$$
with probability at least $1-2\delta$. Therefore,
$$|\wh{\Omega}_{jj}^{-1}-\Omega_{jj}^{-1}|^2\leq C_2\left(\epsilon^2+\left(\frac{s\log\left(\frac{ep}{s}\right)}{n}\right)^2+\frac{\log(1/\delta)}{n}\right),$$
with probability at least $1-4\delta$. Combing the bounds above, we have
\begin{eqnarray*}
\|\wh{\Omega}_{-j,j}-\Omega_{-j,j}\|_1^2 &=& \left\|\wh{\Omega}_{jj}\hat{\beta}_{(j)}-\Omega_{jj}\beta_{(j)}\right\|_1^2 \\
&\leq& 2|\wh{\Omega}_{jj}|^2\|\hat{\beta}_{(j)}-\beta_{(j)}\|^2_1 + 2\|\beta_{(j)}\|_1^2|\wh{\Omega}_{jj}-\Omega_{jj}|^2 \\
&\leq& C_3\left(\frac{s^2\log\left(\frac{ep}{s}\right)}{n}\vee s\epsilon^2+\frac{s\log(1/\delta)}{n}\right),
\end{eqnarray*}
with probability at least $1-4\delta$. Therefore,
$$\|\wh{\Omega}_{*j}-\Omega_{*j}\|_1^2\leq C_4\left(\frac{s^2\log\left(\frac{ep}{s}\right)}{n}\vee s\epsilon^2+\frac{s\log(1/\delta)}{n}\right),$$
with probability at least $1-4\delta$. Finally, a union bound argument gives
$$\|\wh{\Omega}-\Omega\|^2_{\ell_1}=\max_{1\leq j\leq p}\|\wh{\Omega}_{*j}-\Omega_{*j}\|_1^2\leq C_4\left(\frac{s^2\log\left(\frac{ep}{s}\right)}{n}\vee s\epsilon^2+\frac{s\log(1/\delta)}{n}\right),$$
with probability at least $1-4p\delta$. Choose $\delta=\exp\left(-C_5(n\epsilon^2+s\log(ep/s))\right)$, and the proof is complete.
\end{proof}

\begin{proof}[Proofs of Theorem \ref{thm:multiple}, Theorem \ref{thm:group} and Theorem \ref{thm:rrr}]
We first state the proof of Theorem \ref{thm:group}. Recall the notation $P_B$. The same argument that leads to (\ref{eq:interface}) gives
$$\D_{\Xi_{2s}}(\wh{B},P_{B})\geq \frac{1}{2}-\frac{\epsilon}{1-\epsilon}-C_1\left(\sqrt{\frac{ms+s\log\left(\frac{ep}{s}\right)}{n}}+\sqrt{\frac{\log(1/\delta)}{2n}}\right).$$
Since $\wh{B}-B\in\Xi_{2s}$, we use Lemma \ref{lem:curve-m} to deduce (\ref{eq:multiple-pred}). The bound (\ref{eq:multiple-frob}) is a direct implication of (\ref{eq:multiple-pred}) by the definition of $\kappa$. This completes the proof of Theorem \ref{thm:multiple}. Setting $s=p$ gives the proof of Theorem \ref{thm:group}. The proof of Theorem \ref{thm:rrr} follows the same argument, and we omit the details.
\end{proof}

\begin{proof}[Proof of Theorem \ref{thm:EC}]
The proof is the same as that of Theorem \ref{thm:multiple}, except that we need to establish a similar curvature result as Lemma \ref{lem:curve-m} for the elliptical distribution. The same argument that leads to (\ref{eq:depth-bound}) gives
$$\mathbb{E}F\left(\frac{\|(A-B)^TX\|}{\|\Gamma^{1/2}(A-B)^TX\|}\|(A-B)^TX\|\right)-F(0)\leq \eta.$$
By the definition $\sigma^2=\opnorm{\Gamma}$, we have
$$\mathbb{E}F\left(\frac{1}{\sigma}\|(A-B)^TX\|\right)-F(0)\leq \eta.$$
Following the proof of Lemma \ref{lem:curve-m}, it is sufficient to show that $g(t)-1/2\geq Ct\mathbb{E}\{1,\sqrt{Y}\}$, where $g(t)=\mathbb{E}F(t\sqrt{Y})$. We outline the main step without repeating all the details that have already been used in the proof of Lemma \ref{lem:curve-m}. The fact that $g(t)\leq \frac{1}{2}+\eta$ for a sufficiently small $\eta$ implies that $t\leq c_1$. Then,
$$g(t)-\frac{1}{2}\geq t\min_{|t|\leq c_1} f(t)\mathbb{E}\min\{1,\sqrt{Y}\}.$$
Under the assumption that $\min_{|t|\leq c_1} f(t)\geq c_2$, the proof is complete.
\end{proof}

\section*{Acknowledgement}

The author thanks Zhao Ren and Haoyang Liu for reading the manuscript and offering insightful suggestions.

\bibliographystyle{plainnat}
\bibliography{Robust}


\end{document}